\documentclass[a4paper,12pt,leqno]{amsart}
\usepackage{latexsym}
\usepackage[all]{xy}

\usepackage{amssymb} 
\usepackage{amsmath} 
\usepackage{amscd}
\usepackage{color}
\usepackage{comment}
\usepackage{mathtools}

\definecolor{gray}{gray}{0.5}

\numberwithin{equation}{section} 

\newtheorem{theorem}{Theorem}[section]
\newtheorem{lemma}[theorem]{Lemma} 
\newtheorem{corollary}[theorem]{Corollary}
\newtheorem{proposition}[theorem]{Proposition} 
 
\newtheorem{remark}[theorem]{Remark}
\newtheorem{example}[theorem]{Example}
\newtheorem{definition}[theorem]{Definition}

\allowdisplaybreaks[3]

\def\C{\mathbb C}
\def\R{\mathbb R}
\def\Q{\mathbb Q}
\def\Z{\mathbb Z}
\def\P{\mathbb P}

\DeclareMathOperator{\Sym}{Sym}
\DeclareMathOperator{\Hom}{Hom}

\DeclareMathOperator{\height}{ht}
\DeclareMathOperator{\Spec}{Spec}
\DeclareMathOperator{\RSpec}{\bf{Spec}}

\DeclareMathOperator{\Pic}{Pic}

\newcommand{\Dyn}{I}
\newcommand{\HR}[1]{\Phi^+_{#1}}
\newcommand{\Hess}[2]{X(#1,#2)}
\newcommand{\poly}[2]{f_{#1,#2}}
\newcommand{\ideal}[2]{I_{#1,#2}}
\newcommand{\ring}[2]{A_{#1,#2}}
\newcommand{\AHess}[1]{\mathfrak{X}(#1)}
\newcommand{\FHess}[1]{\mathfrak{X}_{{\rm reg}}(#1)}
\newcommand{\pr}{p}
\newcommand{\LB}{L}
\newcommand{\IS}{\mathcal{L}}

\newcommand{\ZHess}[2]{\mathcal{Z}(#1,#2)}
\newcommand{\ZFHess}[1]{\mathcal{Z}_{{\rm reg}}(#1)}

\newcommand{\ZFHessline}[1]{\mathcal{Z}'_{{\rm reg}}(#1)}

\newcommand{\sect}[1]{\sigma_{#1}}
\newcommand{\sectseq}{\sigma'}
\newcommand{\Fsect}{\sigma}
\newcommand{\indshift}{\theta}
\begin{document}
  
\title[Regular Hessenberg varieties]{Geometry of regular Hessenberg varieties}
\author {Hiraku Abe}
\address{Osaka City University Advanced Mathematical Institute, 3-3-138 Sugimoto, Sumiyoshi-ku, Osaka 558-8585, 
Japan}
\email{hirakuabe@globe.ocn.ne.jp}

\author {Naoki Fujita}
\address{Department of Mathematics, Tokyo Institute of Technology, 2-12-1 Oh-okayama, Meguro-ku, Tokyo 152-8551, Japan}
\email{fujita.n.ac@m.titech.ac.jp}

\author {Haozhi Zeng}
\address{School of Mathematical Sciences, Fudan University, Shanghai, 200433, P.R. China}
\email{zenghaozhi@icloud.com}



\begin{abstract}
Let $\mathfrak{g}$ be a complex semisimple Lie algebra. For a regular element $x$ in $\mathfrak{g}$ and a Hessenberg space $H\subseteq \mathfrak{g}$, we consider a regular Hessenberg variety $\Hess{x}{H}$ in the flag variety associated with $\mathfrak{g}$. We take a Hessenberg space so that $\Hess{x}{H}$ is irreducible, and show that the higher cohomology groups of the structure sheaf of $\Hess{x}{H}$ vanish. We also study the flat family of regular Hessenberg varieties, and prove that the scheme-theoretic fibers over the closed points are reduced. We include applications of these results as well.
\end{abstract}

\maketitle

\setcounter{tocdepth}{1}

\section{Introduction}
Hessenberg varieties are closed subvarieties of the flag variety, which were originally introduced in \cite{DMS} and \cite{DMPS}.
They have been studied from various points of view such as geometry, topology, combinatorics, and representation theory (e.g.\ \cite{br-ca04}, \cite{ty}, \cite{MbTym}, \cite{ha-ho-ma}, \cite{Balibanu}, \cite{AHHM}, \cite{AHM} \cite{InskoPrecup}, \cite{dr17}). In Lie type $A$, Tymoczko constructed representations of the symmetric groups on the cohomology rings of regular semisimple Hessenberg varieties (\cite{tymo08}). In 2015, Brosnan-Chow (\cite{Brosnan-Chow}) proved a beautiful conjecture given by Shareshian-Wachs (\cite{sh-wa11}, \cite{sh-wa14}) which connects the representations of the symmetric groups on the cohomology rings of regular semisimple Hessenberg varieties and the quasi-symmetric chromatic functions of graphs constructed from Hessenberg functions. In their argument, regular Hessenberg varieties play an important role, and motivated by their work, Precup (\cite{precup16}) proved that the Betti numbers of an arbitrary regular Hessenberg variety are palindromic in general Lie type. In this paper, we investigate geometry of regular Hessenberg varieties.

To be more precise, let $G$ be a simply connected semisimple algebraic group over $\C$ and $B\subseteq G$ a Borel subgroup of $G$. Let $\mathfrak{g}$ and $\mathfrak{b}$ be the Lie algebras of $G$ and $B$, respectively. A Hessenberg space $H$ is defined to be a linear subspace of $\mathfrak{g}$ that contains $\mathfrak{b}$ and is stable under the adjoint action of $\mathfrak{b}$. For an element $x\in\mathfrak{g}$, the  Hessenberg variety $\Hess{x}{H}\subseteq G/B$ associated with $x$ and $H$ is a (possibly non-irreducible) closed subvariety of $G/B$ given by
\begin{align*}
\Hess{x}{H} = \{gB\in G/B \mid \text{Ad}(g^{-1})x\in H \}.
\end{align*}
A regular Hessenberg variety is the one defined by a regular element $x$ in $\mathfrak{g}$, namely an element $x\in\mathfrak{g}$ whose Lie algebra centralizer $Z_{\mathfrak{g}}(x)$ has the minimum possible dimension. There are two cases of regular Hessenberg varieties that are well-studied: 
(1) if $s\in\mathfrak{g}$ is a regular semisimple element, then $\Hess{s}{H}$ is called a regular semisimple Hessenberg variety; (2) if $N_0\in\mathfrak{g}$ is a regular nilpotent element, then $\Hess{N_0}{H}$ is called a regular nilpotent Hessenberg variety. 
For example, one can choose a Hessenberg space $H$ so that $\Hess{s}{H}$ is a toric variety and $\Hess{N_0}{H}$ is the Peterson variety (see Section \ref{sec: regular Hessenberg varieties} for details). Throughout this paper, we will impose an assumption on the Hessenberg space $H$ so that corresponding regular Hessenberg varieties are irreducible (see condition \eqref{eq:containing simple roots} in Section \ref{sec: regular Hessenberg varieties}). 

We now state our main results. Regular Hessenberg varieties are singular and non-normal in general (\cite{Kostant}, \cite{InskoYong}). Among them, a regular semisimple Hessenberg variety $\Hess{s}{H}$ is known to be smooth and it admits a torus action with finite fixed points (\cite{DMPS}). So it follows that the higher cohomology groups of its structure sheaf vanish (\cite{CL}). As the first main result in this paper, we generalize this fact for an arbitrary regular Hessenberg variety $\Hess{x}{H}$ under the above assumption on $H$.

\begin{theorem}\label{intro: vanishing thm}
Let $X=\Hess{x}{H}$ be a regular Hessenberg variety. Then we have 
\begin{align*}
H^i(X,\mathcal{O}_X)=0 \quad (i>0), 
\end{align*}
where $\mathcal{O}_X$ is the structure sheaf of $X$.
\end{theorem}

To prove this, we first determine the local defining ideals (equivalently, the local coordinate rings) of regular Hessenberg varieties in $G/B$ as generalizations of \cite[Theorem 6]{InskoYong} and \cite[Proposition 3.5]{ab-de-ga-ha}, and then we combine that with the Borel-Weil-Bott theorem for the cohomology groups of line bundles over $G/B$. We will see how the Borel-Weil-Bott theorem controls the cohomology groups of the structure sheaves of regular Hessenberg varieties. As a corollary of Theorem \ref{intro: vanishing thm}, we obtain an isomorphism $\Pic(\Hess{x}{H}) \cong H^2(\Hess{x}{H};\Z)$ between the Picard group and the singular cohomology group of degree $2$. Also, we will see that the determination of the local defining ideals of regular Hessenberg varieties in fact implies that their classes in the $K$-theory of coherent sheaves on $G/B$ do not depend on the choice of regular elements.

In this paper, we also study the family of regular Hessenberg varieties:
\begin{align*}
\FHess{H} = \{(gB,x) \in G/B \times \mathfrak{g}_{{\rm reg}} \mid \text{Ad}(g^{-1})x \in H \},
\end{align*}
where $\mathfrak{g}_{{\rm reg}}$ is the open subset of $\mathfrak{g}$ consisting of regular elements.
We denote the second projection by $\pi\colon \FHess{H} \rightarrow \mathfrak{g}_{{\rm reg}}$, and then each fiber $\pi^{-1}(x)$ for $x\in\mathfrak{g}_{{\rm reg}}$ is the regular Hessenberg variety $\Hess{x}{H}$. As we will see, the projection $\pi\colon \FHess{H} \rightarrow \mathfrak{g}_{{\rm reg}}$ is a flat morphism. Now the following is the second main result of this paper.
\begin{theorem}\label{intro: reducedness thm}
The scheme-theoretic fibers of the morphism \emph{$\pi\colon\FHess{H} \rightarrow \mathfrak{g}_{{\rm reg}}$} over the closed points are reduced.
\end{theorem}

This generalizes the reducedness of scheme-theoretic fibers of a (1-parameter) flat degeneration of a regular semisimple Hessenberg variety $\Hess{s}{H}$ to a regular nilpotent Hessenberg variety $\Hess{N_0}{H}$, studied in \cite{an-ty} and \cite{ab-de-ga-ha} in Lie type $A$.

As a corollary of Theorem \ref{intro: reducedness thm}, we compute the volume of the projective embeddings of $\Hess{x}{H}$ with respect to the very ample line bundles defined over $G/B$. More precisely, let $V_{\lambda}$ be the irreducible representation of $G$ with highest weight $\lambda$. We have the standard embedding of $G/B$ into $\P(V_{\lambda})$ (see the end of Section \ref{subsec: reducedness of fibers} for details). By composing this with the inclusion $\Hess{x}{H}\hookrightarrow G/B$, we obtain a closed embedding $\Hess{x}{H}\hookrightarrow \P(V_{\lambda})$. Theorem \ref{intro: reducedness thm} now shows that the Hilbert polynomial of this embedding does not depend on a choice of $x\in\mathfrak{g}_{{\rm reg}}$, and we provide a Lie-theoretic formula for the volume of this embedding by using this result. This is a generalization of a similar computation for regular nilpotent Hessenberg varieties $\Hess{N_0}{H}$ in Lie type A given in \cite{ab-de-ga-ha}.

\bigskip
\noindent \textbf{Acknowledgements.} We are grateful to William Slofstra for the fruitful discussion on a Lie-theoretic approach to sheaf cohomology groups of Hessenberg varieties. 
We appreciate Patrick Brosnan for kindly explaining the flatness of the family of the regular Hessenberg varieties. We also thank Mikiya Masuda for his support and encouragement. We are also grateful to Alexander Woo for explaining the formula for K-theory classes of regular Hessenberg varieties.
Part of the research for this paper was carried out at the Fields Institute; the authors would like to thank the institute for its hospitality. The first author is partially supported by a JSPS Grant-in-Aid for Young Scientists (B): 15K17544 and a JSPS Research Fellowship for Young Scientists Postdoctoral Fellow: 16J04761. The second author is partially supported by a JSPS Research Fellowship for Young Scientists Doctoral Course Students: 16J00420. The third author was partially supported by NSFC, grants No. 11661131004 and 11431009.

\bigskip
\section{Background and notation}\label{sec: background and notation}
\subsection{Regular Hessenberg varieties}\label{sec: regular Hessenberg varieties}

Let $G$ be a simply connected, semisimple algebraic group over $\C$ of rank $n$. Choose a Borel subgroup $B\subset G$ and a maximal torus $T\subset B$. Let $W=N(T)/T$ be the Weyl group, where $N(T)$ is the normalizer of $T$ in $G$. We denote by $\mathfrak{g}$, $\mathfrak{b}$, and $\mathfrak{t}$ the Lie algebra of $G$, $B$, and $T$, respectively. We have a root space decomposition $\mathfrak{g}=\mathfrak{t}\oplus\bigoplus_{\alpha\in\Phi}\mathfrak{g}_{\alpha}$, where $\Phi$ is the set of roots. Let $\Phi^+$ (resp.\ $\Phi^-$) be the set of positive (resp.\ negative) roots and $\Delta=\{\alpha_1,\ldots,\alpha_n\}\subset\Phi^+$ the set of simple roots. We write $\mathfrak{g} = \mathfrak{t} \oplus \mathfrak{n} \oplus \mathfrak{n}^-$ by putting $\mathfrak{n}=\bigoplus_{\alpha\in\Phi^+}\mathfrak{g}_{\alpha}$ and $\mathfrak{n}^-=\bigoplus_{\alpha\in\Phi^+}\mathfrak{g}_{-\alpha}$. Denote by $I$ the Dynkin diagram of $\mathfrak{g}$ which we identify with the indexing set $\{1,\ldots,n\}$ for the simple roots $\alpha_1,\ldots,\alpha_n$ unless otherwise specified. For $\alpha,\beta\in\Phi^+$, we write $\alpha \preceq \beta$ when $\alpha\in\beta-\textstyle{\sum_{i\in \Dyn}}\Z_{\geq0}\alpha_i$.

A \textbf{Hessenberg space} $H$ is a $\mathfrak{b}$-submodule $H\subseteq \mathfrak{g}$ containing $\mathfrak{b}$ (\cite{DMPS}). The $\mathfrak{b}$-invariance of $H$ implies that we can always write
\begin{align}\label{eq: decomp of H}
H = \mathfrak{b} \oplus \bigoplus_{\alpha\in \HR{H}}\mathfrak{g}_{-\alpha}
\end{align}
for some subset $\HR{H}\subseteq \Phi^+$ which satisfies the following property: if $\beta\in\HR{H}$ and $\alpha$ is a positive root satisfying $\alpha\preceq\beta$, then we must have $\alpha\in\HR{H}$. Throughout this paper, we assume that $\HR{H}$ contains all the simple roots:
\begin{align}\label{eq:containing simple roots}
\Delta \subseteq \HR{H}. \tag{$\dagger$}
\end{align}

An element $x\in\mathfrak{g}$ is called \textbf{regular} if the Lie algebra centralizer $Z_{\mathfrak{g}}(x)$ has the minimum possible dimension, i.e.\ the rank $n$ of $\mathfrak{g}$. In other words, $x\in\mathfrak{g}$ is regular if and only if its adjoint orbit in $\mathfrak{g}$ has the maximum possible dimension (\cite[Ch.\ 1]{Humphreys95}).

Let $x\in\mathfrak{g}$ be a regular element and $H\subseteq \mathfrak{g}$ a Hessenberg space which satisfies condition \eqref{eq:containing simple roots}. 
The \textbf{regular Hessenberg variety associated to $x$ and $H$} is defined to be 
\begin{align}\label{eq:def of reg Hess}
\Hess{x}{H} \coloneqq \{gB \in G/B \mid \text{Ad}(g^{-1})x \in H \}.
\end{align}
Namely, an element $gB\in G/B$ lies in $\Hess{x}{H}$ if and only if the $\mathfrak{g}_{-\alpha}$-component of $\text{Ad}(g^{-1})x$ vanishes for all $\alpha\in \Phi^+ \setminus \HR{H}$ with respect to the root space decomposition of $\mathfrak{g}$. When $H$ is the Lie algebra $\mathfrak{g}$ itself, we always have $\Hess{x}{\mathfrak{g}}=G/B$. In Lie type $A$, geometry of the family of regular Hessenberg varieties plays an important role in the paper \cite{Brosnan-Chow} of Brosnan-Chow to prove the Shareshian-Wachs conjecture  which is thought of as one of the main steps toward proving the Stanley-Stembridge conjecture in graph theory (see \cite{sh-wa14} for the details). Also, Precup \cite{precup16} proved that the Betti numbers of an arbitrary regular Hessenberg variety are palindromic. The class of regular Hessenberg varieties contains the following two classes which are particularly well-studied:

(1) If $s$ is a regular semisimple element of $\mathfrak{g}$, then $\Hess{s}{H}$ is called a \textbf{regular semisimple Hessenberg variety} (\cite{DMPS}). It is a non-singular projective variety, and it is preserved by the action of a maximal torus of $G$ on $G/B$, where the fixed point set of $\Hess{s}{H}$ is naturally identified with the Weyl group $W$. If $\HR{H}=\Delta$ in addition, then it is a toric variety (\cite{klya}, \cite{proc90}, \cite{DMPS}). In Lie type $A$, this toric variety is also called a permutohedral variety.

(2) If $N_0$ is a regular nilpotent element of $\mathfrak{g}$, then $\Hess{N_0}{H}$ is called a \textbf{regular nilpotent Hessenberg variety}. It is a singular (and non-normal) projective variety in general (\cite{Kostant}, \cite{InskoYong}). The cohomology ring $H^*(\Hess{N_0}{H};\R)$ has been studied well (\cite{br-ca04}, \cite{dr}, \cite{ha-ho-ma}, \cite{AHHM}), and recently Sommers-Tymoczko (\cite{so-ty}) and Abe-Horiguchi-Masuda-Murai-Sato (\cite{AHMMS}) discovered a connection with the theory of hyperplane arrangements. If $\HR{H}=\Delta$ in addition, then it is called the \textbf{Peterson variety} (\cite{Kostant}, \cite{ha-ho-ma}, \cite{Balibanu}) which appeared in Kostant's study of the quantum cohomology of the flag variety in \cite{Kostant}. 

For an arbitrary regular element, the following holds. Recall that $\HR{H}\subseteq \Phi^+$ is a set of positive roots defined in \eqref{eq: decomp of H}.
\begin{theorem}\label{thm:precup}
$($Precup \cite{precup16}$)$
For a regular element $x$ and a Hessenberg space satisfying condition \eqref{eq:containing simple roots}, 
the regular Hessenberg variety $\Hess{x}{H}$ is irreducible of dimension $|\HR{H}|$. In particular, the codimension of $\Hess{x}{H}$ in $G/B$ is $|\Phi^+\setminus\HR{H}|$.
\end{theorem}

Throughout this paper, we fix non-zero vectors $E_{\alpha}\in\mathfrak{g}_{\alpha}$ and $F_{\alpha}\in\mathfrak{g}_{-\alpha}$ for each positive root $\alpha\in\Phi^+$ so that $\mathfrak{g}_{\alpha}=\C E_{\alpha}$ and $\mathfrak{g}_{-\alpha}=\C F_{\alpha}$. 
Recall that $I$ is the Dynkin diagram of the root system $\Phi$. An arbitrary regular element $x\in\mathfrak{g}$ is conjugated to an element of the form (\cite[Proposition 0.4]{Kostant63})
\begin{align}\label{eq:Jordan decomp}
x_J = s_J + n_J
\end{align}
for some $J\subseteq I$ (regarded as a full-subgraph), where $n_J=\sum_{i\in J}E_{\alpha_i}\in\mathfrak{n}$ and the semisimple element $s_J\in\mathfrak{t}$ satisfies
\begin{align}\label{eq:centralizer of s_J}
Z_{\mathfrak{g}}(s_J) = \mathfrak{g}(J) \oplus Z.
\end{align}
Here, $\mathfrak{g}(J)$ is the complex semisimple Lie algebra corresponding to the Dynkin diagram $J$, and $Z\subseteq\mathfrak{t}$ is the center of $Z_{\mathfrak{g}}(s_J)$.

For example, when $J=\emptyset$, the element $x_{\emptyset}(=s_{\emptyset})$ is regular semisimple, and when $J=I$, the element $x_I(=n_I)$ is regular nilpotent.

\vspace{5pt}
\begin{example}\label{eg: type A_2}
\emph{
Let $G=\text{SL}_3(\C)$ and $B$ the subgroup of $\text{SL}_3(\C)$ consisting of upper-triangular matrices. 
We take the standard maximal torus $T\subset G$ whose elements are diagonal matrices.
Let $\alpha_1,\alpha_2$ be the simple roots corresponding to the homomorphisms $t=\text{diag}(t_1,t_2,t_3)(\in T)\mapsto t_1t_2^{-1}(\in\C^{\times})$ and $t=\text{diag}(t_1,t_2,t_3)(\in T)\mapsto t_2t_3^{-1}(\in\C^{\times})$, respectively.
Then we have $\Phi^+=\{\alpha_1,\alpha_2,\alpha_1+\alpha_2\}$, and let
\begin{align*}
E_{\alpha_1}
=
\begin{pmatrix}
0 & 1 & 0  \\
0 & 0 & 0 \\
0 & 0 & 0
\end{pmatrix}
, \ 
E_{\alpha_2}
=
\begin{pmatrix}
0 & 0 & 0  \\
0 & 0 & 1 \\
0 & 0 & 0
\end{pmatrix}
, \ 
E_{\alpha_1+\alpha_2}
=
\begin{pmatrix}
0 & 0 & 1  \\
0 & 0 & 0 \\
0 & 0 & 0
\end{pmatrix}.
\end{align*}
If we take $J=\{1\}\subset\{1,2\}$, then we have
\begin{align}\label{eq: example of regular element}
x_J
=
\begin{pmatrix}
\mu_1 & 1 & 0  \\
0 & \mu_1 & 0 \\
0 & 0 & \mu_2
\end{pmatrix}
=
\begin{pmatrix}
\mu_1 & 0 & 0  \\
0 & \mu_1 & 0 \\
0 & 0 & \mu_2
\end{pmatrix}
+
\begin{pmatrix}
0 & 1 & 0  \\
0 & 0 & 0 \\
0 & 0 & 0
\end{pmatrix}
\end{align}
for some distinct complex numbers $\mu_1\neq \mu_2$.
}
\end{example}

\subsection{Regular Hessenberg varieties as zeros of sections}\label{subsec: zero scheme}
Let $x\in\mathfrak{g}$ be a regular element. Since a Hessenberg space $H\subseteq \mathfrak{g}$ is a $B$-submodule, the quotient space $\mathfrak{g}/H$ admits a $B$-action given by $b\cdot \overline{X}=\overline{\text{Ad}(b)X}$ for $b\in B$ and $X\in\mathfrak{g}$, where $\overline{X}\in\mathfrak{g}/H$ denotes the image of $X$. So we have a vector bundle $G\times_B(\mathfrak{g}/H)$ over $G/B$, which is defined by the quotient of the direct product $G\times(\mathfrak{g}/H)$ by the right $B$-action given by $(g,\overline{X})\cdot b=(gb,\overline{\text{Ad}(b^{-1})X})$ for $(g,\overline{X})\in G\times(\mathfrak{g}/H)$ and $b\in B$. Then the regular element $x$ induces a well-defined section
\begin{align}\label{eq: section of the Hessenberg vector bundle}
\sect{x} \colon G/B \rightarrow G\times_B(\mathfrak{g}/H) 
\quad ; \quad 
gB \mapsto [g,\overline{\text{Ad}(g^{-1})x}].
\end{align}
It is clear that the regular Hessenberg variety $\Hess{x}{H}$ is the zero set $\mathbb{V}(\sect{x})$. Regarding $G/B$ as an integral scheme, we can consider the zero scheme of the section $\sect{x}$ (cf.\ \cite[Appendix B.3.2]{Fulton}). 

\begin{definition}
Let $\mathcal{Z}(x,H)$ be the zero scheme of the section $\sect{x}$ given at \eqref{eq: section of the Hessenberg vector bundle}.
\end{definition}

Set-theoretically $\mathcal{Z}(x,H)$ is the same as the zero set $\mathbb{V}(\sect{x})$, but it carries the multiplicities of the local equations given by the section $\sect{x}$. Since the rank of the vector bundle $G\times_B(\mathfrak{g}/H)$ appearing in \eqref{eq: section of the Hessenberg vector bundle} is equal to $|\Phi^+\setminus\HR{H}|$, Theorem \ref{thm:precup} above shows that the zero scheme $\mathcal{Z}(x,H)$ has the expected codimension. That is, in the ambient space $G/B$, it is locally cut out by $\dim G/B - \dim \mathcal{Z}(x,H)$ functions. Thus, by \cite[Lemma A.7.1]{Fulton}, the following property of $\mathcal{Z}(x,H)$ follows (see \cite[Sect.\ 18.5]{Eisenbud} for the definition of locally complete intersection rings).

\begin{proposition}\label{prop: LCI and CM}
The scheme $\mathcal{Z}(x,H)$ is a local complete intersection, and hence a Cohen-Macaulay scheme.
\end{proposition}

In the next section, we will show that the scheme $\mathcal{Z}(x_J,H)$ is reduced, and we will determine local defining ideals of the regular Hessenberg variety $\Hess{x_J}{H}$ in $G/B$ as its corollary.

\bigskip
\section{Defining ideals}\label{sec: defining ideals}
Let $x_J\in\mathfrak{g}$ be a regular element of the form \eqref{eq:Jordan decomp} and $H$ a Hessenberg space satisfying condition \eqref{eq:containing simple roots}. In this section, we will show that the zero scheme $\mathcal{Z}(x_J,H)$ is reduced, and as a corollary we will determine the defining ideals of the regular Hessenberg variety $\Hess{x_J}{H}$ with respect to a standard affine open cover of $G/B$.

We start from constructing a local description of $\mathcal{Z}(x_J,H)$. Let $B^-\subset G$ be the opposite Borel subgroup with respect to $B$ and $U^-\subset B^-$ its unipotent radical. Then $U^-$ is a affine closed subvariety of $G$, and we have an open embedding 
\begin{align*}
i \colon U^- \hookrightarrow  G/B 
\quad ; \quad 
g \mapsto gB.
\end{align*}
For each $w\in W=N(T)/T$, let us fix a representative $\tilde{w}\in N(T)$ such that $w=[\tilde{w}]$ in $W$. We write
\begin{align}\label{eq:def of wU-}
wU^- \coloneqq \{ \tilde{w} gB\in G/B \mid g\in U^- \} \subset G/B.
\end{align}
This is an affine open subvariety of $G/B$ which is isomorphic to $U^-$. Note that the subset $wU^-\subset G/B$ does not depend on a choice of the representative $\tilde{w}$ of $w$ since $U^-$ is invariant under the adjoint action of $T$. It follows that these affine open subvarieties cover $G/B$:
\begin{align}\label{eq: open cover of G/B}
G/B = \bigcup_{w\in W} wU^-.
\end{align}
Using a choice of the representative $\tilde{w}$ of $w$, we have a section $\phi_{\tilde{w}}$ of the projection map $\pr \colon G\rightarrow G/B$ over $wU^-$, given by
\begin{align}\label{eq: embedding of wU^-}
\phi_{\tilde{w}} \colon wU^- \rightarrow  G
\quad ; \quad 
gB \mapsto \tilde{w} \cdot i^{-1}(\tilde{w}^{-1} gB).
\end{align}
In other words, we have
\begin{align}\label{eq: embedding of wU^- 2}
\phi_{\tilde{w}}(\tilde{w} gB)=\tilde{w} g \quad \text{for } g\in U^-
\end{align}
under the presentation \eqref{eq:def of wU-}. Thus, the principal $B$-bundle $p\colon G\rightarrow G/B$ and hence the vector bundle $G\times_B(\mathfrak{g}/H)$ are trivialized over $wU^-$. Namely, the section $\sect{x_J}$ from \eqref{eq: section of the Hessenberg vector bundle} on $wU^-$ is identified with the morphism
\begin{equation*}
\sect{x_J} \colon wU^- \rightarrow wU^-\times (\mathfrak{g}/H)
\quad; \quad 
gB  \mapsto (gB,\overline{\text{Ad}((\phi_{\tilde{w}}(gB))^{-1})x_J}).
\end{equation*}
So we can describe the section $\sect{x_J}$ on $wU^-$ as a collection of regular functions on $wU^-$ via the decomposition
\begin{equation*}
\mathfrak{g}/H \cong \bigoplus_{\alpha\in\Phi^+ \setminus \HR{H}} \mathfrak{g}_{-\alpha}.
\end{equation*}
To do that, recall from Section \ref{sec: regular Hessenberg varieties} that we chose a non-zero vector $F_{\alpha}\in\mathfrak{g}_{-\alpha}$ for each positive root $\alpha\in\Phi^+$ so that $\mathfrak{g}_{-\alpha}=\C F_{\alpha}$, and define a function $\poly{\alpha}{\tilde{w}} \colon wU^-\rightarrow \C F_{\alpha}\cong\C$ by
\begin{align}\label{eq:def of f_alpha}
\poly{\alpha}{\tilde{w}}(gB) \coloneqq (\text{Ad}((\phi_{\tilde{w}}(gB))^{-1})x_J)_{-\alpha}
\quad \text{for } gB\in wU^-, 
\end{align}
where $\phi_{\tilde{w}}$ is the section given in \eqref{eq: embedding of wU^-} and we identify $\C F_{\alpha}\cong \C$ by $tF_{\alpha}\mapsto t$. Then it is clear that the section $\sect{x_J}$ on $wU^-$ is a collection of the functions $\poly{\alpha}{\tilde{w}}$ for $\alpha\in\Phi^+ \setminus \HR{H}$. Namely, the intersection $\Hess{x_J}{H}\cap wU^-$ is the zero set of these functions. To describe the local presentation of the zero scheme $\mathcal{Z}(x_J,H)$ on $wU^-$, we then make the following definition. Let $\C[wU^-]$ be the coordinate ring of the affine variety $wU^-$.

\begin{definition}\label{def:ideal and ring}
For $w\in W$, let
\begin{align*}
\ideal{H}{w} \coloneqq (\poly{\alpha}{\tilde{w}} \mid \alpha\in \Phi^+ \setminus \HR{H} ) \subseteq \C[wU^-]
\end{align*}
be the ideal of $\C[wU^-]$ generated by $\poly{\alpha}{\tilde{w}}$ defined in \eqref{eq:def of f_alpha}, and we set 
\begin{align*}
\ring{H}{w}\coloneqq\C[wU^-]/\ideal{H}{w}.
\end{align*}
\end{definition}

Note that the space of closed points of $\Spec \ring{H}{w}$ is homeomorphic to the intersection $\Hess{x_J}{H}\cap wU^-$. The following is an immediate consequence of the construction of the zero scheme $\mathcal{Z}(x_J,H)$.

\begin{lemma}\label{lem: affine cover for zero}
We have an affine open cover
\begin{align*}
\mathcal{Z}(x_J,H) = \bigcup_{w\in W} \Spec \ring{H}{w}.
\end{align*}
\end{lemma}

The goal of this section is to show that the ring $\ring{H}{w}$ is the coordinate ring of $\Hess{x_J}{H}\cap wU^-$. For this, we need to prove that $\ring{H}{w}$ is reduced (i.e.\ we need to show that $\ideal{H}{w}$ is a radical ideal).

\vspace{10pt}
For the convenience of the argument in what follows, let us make an enumeration of positive roots:
\begin{align}\label{eq:enumeration}
\Phi^+ = \{\beta_1,\ldots,\beta_N\},
\end{align}
where $N\coloneqq|\Phi^+|$ is the number of positive roots. Then we have an isomorphism
\begin{align*}
\C^N \xrightarrow{\sim}  U^-
\quad ; \quad 
(t_1,\ldots,t_N) \mapsto \exp(t_1F_{\beta_1}) \cdots \exp(t_N F_{\beta_N}).
\end{align*}
To prove that the ring $\ring{H}{w}$ is reduced for all $w\in W$, we first focus on the case for which $w$ is the longest element $w_0=[\tilde{w}_0]\in W=N(T)/T$, where $\tilde{w}_0\in N(T)$ is the chosen representative to construct the section given in \eqref{eq: embedding of wU^-}. The above coordinate system on $U^-$ induces a coordinate system on $w_0U^-$ by
\begin{align*}
\C^N \xrightarrow{\sim}  w_0U^-
\quad ; \quad 
(t_1,\ldots,t_N) \mapsto \tilde{w}_0\exp(t_1F_{\beta_1}) \cdots \exp(t_N F_{\beta_N})B.
\end{align*}
In particular, the coordinate ring $\C[w_0U^-]$ of $w_0U^-$ is a polynomial ring over $\C$ with $N$ variables.

\vspace{10pt}
Now the definition \eqref{eq:def of f_alpha} applied to $\poly{\alpha}{\tilde{w}_0} $ together with \eqref{eq: embedding of wU^- 2} shows that $\poly{\alpha}{\tilde{w}_0}$ is a function on $w_0U^-$ which sends an element $\tilde{w}_0\exp(t_1F_{\beta_1}) \cdots \exp(t_N F_{\beta_N})B\in w_0U^-$ to 
\begin{align}\label{eq: starting form of fw0}
(\text{Ad}((\tilde{w}_0\exp(t_1F_{\beta_1}) \cdots \exp(t_N F_{\beta_N}))^{-1})x_J)_{-\alpha}.
\end{align}

\vspace{5pt}
\begin{example}
\emph{
Let $G=\text{SL}_3(\C)$, and we use the same notations from Example \ref{eg: type A_2}. The Weyl group $W$ in this case is isomorphic to the permutation group $\mathfrak{S}_3$ on three letters. As for the representative of the longest element $w_0\in\mathfrak{S}_3$ in $\text{SL}_3(\C)$, one can take 
\begin{align*}
\tilde{w}_0
=
\begin{pmatrix}
0 & 0 & 1  \\
0 & 1 & 0 \\
-1 & 0 & 0
\end{pmatrix}
\in \text{SL}_3(\C).
\end{align*}
Let us use the enumeration of $\Phi^+=\{\alpha_1,\alpha_2,\alpha_1+\alpha_2\}$ given by
\begin{align*}
\beta_1=\alpha_1, \ \beta_2=\alpha_2, \ \beta_3=\alpha_1+\alpha_2.
\end{align*}
Then the above coordinate system on $w_0U^-$ is described as follows:
\begin{align*}
w_0U^- =
\left\{
\left.
\begin{pmatrix}
t_3 & t_2 & 1  \\
t_1 & 1 & 0 \\
-1 & 0 & 0
\end{pmatrix}
B\in \text{SL}_3(\C)/B \ 
\right| \
t_1,t_2,t_3\in\C
\right\}.
\end{align*}
Let $x_J$ be a regular element of $\mathfrak{sl}_3(\C)$ given by \eqref{eq: example of regular element}. Then the function $\poly{\alpha_1+\alpha_2}{\tilde{w}_0}$ as a polynomial in $t_1, t_2, t_3$ is given by
\begin{align*}
\poly{\alpha_1+\alpha_2}{\tilde{w}_0}(t_1,t_2,t_3) 
&=\left(
\begin{pmatrix}
t_3 & t_2 & 1  \\
t_1 & 1 & 0 \\
-1 & 0 & 0
\end{pmatrix}^{-1}
\begin{pmatrix}
\mu_1 & 1 & 0  \\
0 & \mu_1 & 0 \\
0 & 0 & \mu_2
\end{pmatrix}
\begin{pmatrix}
t_3 & t_2 & 1  \\
t_1 & 1 & 0 \\
-1 & 0 & 0
\end{pmatrix}
\right)_{3,1} \\
&
= (\mu_1-\mu_2)t_3 + t_1 - (\mu_1-\mu_2)t_1t_2.
\end{align*}
}
\end{example}

\vspace{10pt}
\begin{lemma}\label{lem:w_0 is polynomial ring}
The ring $\ring{H}{w_0}$ is isomorphic to a polynomial ring over $\C$, and hence it is reduced.
\end{lemma}

To give a proof of this Lemma, we use the following property of the partial order $\preceq$ on $\Phi^+$ which one can prove by examining each irreducible root system (cf.\ \cite[Sect.\ 3.20]{Humphreys90}). Recall that we write $\alpha \preceq \beta$ for $\alpha,\beta\in\Phi^+$ when $\alpha\in\beta-\textstyle{\sum_{i\in \Dyn}}\Z_{\geq0}\alpha_i$.

\begin{lemma}\label{lem:injective map}
There exists an injective map $\iota \colon \{\alpha\in\Phi^+\mid \height(\alpha)\geq2\} \hookrightarrow \Phi^+$ such that $\iota(\alpha)\prec\alpha$ and $\height(\iota(\alpha))=\height(\alpha)-1$ for each $\alpha\in\Phi^+$ with $\height(\alpha)\geq2$.
\end{lemma}

We now prove Lemma \ref{lem:w_0 is polynomial ring}.

\begin{proof}[Proof of Lemma~$\ref{lem:w_0 is polynomial ring}$]
Let $\alpha\in\Phi^+ \setminus \HR{H}$. From \eqref{eq: starting form of fw0}, we have that
\begin{align}\notag
&\poly{\alpha}{\tilde{w}_0} (\tilde{w}_0\exp(t_1F_{\beta_1}) \cdots \exp(t_N F_{\beta_N})B) \\ \notag
&\quad=
(\text{Ad}((\tilde{w}_0\exp(t_1F_{\beta_1}) \cdots \exp(t_NF_{\beta_N}))^{-1})x_J)_{-\alpha} \\ \notag
&\quad=
(\text{Ad}(\exp(-t_NF_{\beta_N}) \cdots \exp(-t_1F_{\beta_1}))(\text{Ad}(\tilde{w}_0^{-1})x_J))_{-\alpha} \\ \label{eq:computation of poly}
&\quad=
(\exp(\text{ad}(-t_NF_{\beta_N})) \cdots \exp(\text{ad}(-t_1F_{\beta_1}))(\text{Ad}(\tilde{w}_0^{-1})x_J))_{-\alpha}.
\end{align}
To deal with the adjoint action of $\tilde{w}_0$, we define $i^*\in[n]$ for each $i\in[n]$ to be a positive integer satisfying $w_0^{-1}\alpha_i=-\alpha_{i^*}$. Recall from \eqref{eq:Jordan decomp} that $x_J$ appearing in \eqref{eq:computation of poly} is written as $x_J=s_J + \sum_{i\in J} E_{\alpha_i}$. So let us write
\begin{align*}
s'_{J}\coloneqq\text{Ad}(\tilde{w}_0^{-1})s_J\in\mathfrak{t} \quad \text{and} \quad  
F'_{\alpha_{i^*}}\coloneqq \text{Ad}(\tilde{w}_0^{-1})E_{\alpha_i}\in\mathfrak{g}_{-\alpha_{i^*}}.
\end{align*}
By expanding the exponentials above, the last expression can be written as
\begin{equation}\label{eq:expansion of f_alpha}
\begin{split}
&\sum_{\substack{a_N,\ldots,a_1\geq 0 ; \\ a_1\beta_1+\cdots+a_N\beta_N=\alpha}}
\frac{(-t_N)^{a_N}}{a_N!} \cdots \frac{(-t_1)^{a_1}}{a_1!}
\text{ad}(F_{\beta_N})^{a_N} \cdots \text{ad}(F_{\beta_1})^{a_1} 
(s'_J) \\
&+
\sum_{i\in J}\sum_{\substack{a_N,\ldots,a_1\geq 0 ; \\ a_1\beta_1+\cdots+a_N\beta_N+\alpha_{i^*}=\alpha}}
\frac{(-t_N)^{a_N}}{a_N!} \cdots \frac{(-t_1)^{a_1}}{a_1!}
\text{ad}(F_{\beta_N})^{a_N} \cdots \text{ad}(F_{\beta_1})^{a_1} 
(F'_{\alpha_{i^*}}).
\end{split}
\end{equation}
This means that $\poly{\alpha}{\tilde{w}_0}$ is a polynomial in variables $t_{\ell}$ for which $\beta_{\ell}\preceq \alpha$. Let $1\leq k\leq N$ be an integer such that $\alpha=\beta_k$. Namely, $\alpha$ is the $k$-th one in the enumeration \eqref{eq:enumeration}. We now take cases. To do that, let us denote by $J^*$ the full-subgraph of $I$ whose vertex set is $\{j^*\in[n]\mid j\in J\}$. We regard $J^*$ as a Dynkin diagram, and we mean by $\Phi(J^*)$ the set of roots of $\mathfrak{g}(J^*)$ which is the complex semisimple Lie algebra corresponding to the Dynkin diagram $J^*$.\\

\noindent
\textbf{Case 1:} $\alpha \notin \Phi(J^*)$.
Observe that the term $(-t_k)\text{ad}(F_{\alpha})(s'_J)$ appears in the first summation in \eqref{eq:expansion of f_alpha}. We have
\begin{align*}
Z_{\mathfrak{g}}(s'_{J})
=
\mathfrak{g}(J^*) \oplus Z'
\end{align*} 
by \eqref{eq:centralizer of s_J}, where $Z'\subseteq\mathfrak{t}$ is the center of $Z_{\mathfrak{g}}(s'_{J})$. Because of the assumption of this case, this shows that $F_{\alpha}\notin Z_{\mathfrak{g}}(s'_{J})$. Thus, the term $(-t_k)\text{ad}(F_{\alpha})(s'_J)$ is non-zero. This means that 
\begin{align*}
\poly{\alpha}{\tilde{w}_0} \in \C^{\times}t_k + (\text{a polynomial in variables $t_{\ell}$ for $1 \leq \ell \leq N$ such that $\beta_{\ell}\prec\alpha$}).
\end{align*} 

\vspace{10pt}
\noindent
\textbf{Case 2:} $\alpha \in \Phi(J^*)$.
In this case, we have $F_{\alpha}\in Z_{\mathfrak{g}}(s'_{J})$, and hence $\poly{\alpha}{\tilde{w}_0}$ is a polynomial in variables $t_{\ell}$ for $1 \leq \ell \leq N$ such that $\beta_{\ell}\prec\alpha$. Observe that $\height(\alpha)\geq 2$ since we have $\alpha\in\Phi^+\setminus\HR{H}$ and $\HR{H}$ contains all the simple roots $\Delta$. So $\iota(\alpha)$ is defined, where $\iota$ is an injective map given in Lemma \ref{lem:injective map}, and we have $\iota(\alpha)\prec\alpha$ and $\height(\iota(\alpha))=\height(\alpha)-1$. So let $\alpha_{j^*}\in\Delta$ be the simple root satisfying $\iota(\alpha)=\alpha-\alpha_{j^*}$. Then we have
\begin{align}\label{eq:beta_ell and alpha_j}
\iota(\alpha) , \ 
\alpha_{j^*} \in \Phi(J^*)
\end{align} 
since $\alpha = \iota(\alpha) + \alpha_{j^\ast}$ and $\alpha$ is a non-negative sum of simple roots contained in $J^*$. Note that $\alpha_{j^*}\neq\iota(\alpha)$ since if the equality holds then we get $\alpha=2\alpha_{j^*}$ which is impossible. The map $\iota$ induces an injective map
\begin{align*}
\indshift \colon \{\ell \mid 1\leq \ell \leq N, \ \height(\beta_{\ell})\geq2\} \hookrightarrow \{1,2,\ldots,N\}
\end{align*}
by claiming $\iota(\beta_{\ell})=\beta_{\indshift(\ell)}$. Then we have $\iota(\alpha)=\beta_{\indshift(k)}$ since $\alpha=\beta_k$. Write $\alpha_{j^*}=\beta_{p}$.

First assume $\indshift(k)<p$. Namely, we assume that $\iota(\alpha)$ appears before $\alpha_{j^*}$ in the enumeration \eqref{eq:enumeration}. We then have 
\begin{equation}\label{eq:case 2 f_alpha}
\begin{split}
&\poly{\alpha}{\tilde{w}_0} (\tilde{w}_0\exp(t_1F_{\beta_1}) \cdots \exp(t_N F_{\beta_N})B) \\
&\quad= 
(-t_{p})(-t_{\indshift(k)})\text{ad}(F_{\alpha_{j^*}})\text{ad}(F_{\iota(\alpha)})(s'_J) 
+
(-t_{\indshift(k)})\text{ad}(F_{\iota(\alpha)})(F'_{\alpha_{j^*}}) \\
&\qquad\qquad +
(\text{a polynomial in variables $t_{\ell}$ for which 
$\beta_{\ell}\prec\alpha$ and $\beta_{\ell}\neq\iota(\alpha)$}).
\end{split}
\end{equation} 
Because of \eqref{eq:beta_ell and alpha_j}, we have $F_{\iota(\alpha)}\in Z_{\mathfrak{g}}(s'_J)$. So we get $\text{ad}(F_{\iota(\alpha)})(s'_J)=0$, which shows that the first summand in \eqref{eq:case 2 f_alpha} vanishes. Since we have $[\C F_{\iota(\alpha)}, \C F'_{\alpha_{j^*}}]=\C F_{\alpha}$, the second summand in \eqref{eq:case 2 f_alpha} satisfies 
\begin{align*}
(-t_{\indshift(k)})\text{ad}(F_{\iota(\alpha)})(F'_{\alpha_{j^*}})
\in \C^{\times} t_{\indshift(k)}.
\end{align*} 
This means that
\begin{align*}
\poly{\alpha}{\tilde{w}_0} \in \C^{\times}t_{\indshift(k)} + (\text{a polynomial in variables $t_{\ell}$ for which 
$\beta_{\ell}\prec\alpha$ and $\beta_{\ell}\neq\iota(\alpha)$}).
\end{align*} 
An argument similar to that used here works for the case $\indshift(k)>p$ as well because of \eqref{eq:beta_ell and alpha_j}, and we obtain the same conclusion.

\vspace{10pt}
Now we prove that the ring $\ring{H}{w_0}=\C[w_0U^-]/\ideal{H}{w_0}$ is isomorphic to a polynomial ring over $\C$. We argued above that if $\alpha$ is in Case 2 then $\iota(\alpha)$ is also in Case 2; in particular $\iota(\alpha)$ cannot be in Case 1. Combining with the injectivity of the map $\iota$, the conclusions of the above two cases show that we can inductively eliminate variables without imposing non-trivial relations among the rest variables. More precisely, let 
\begin{align*}
D = \{k \mid \text{$\beta_k$ belongs to Case 1} \} \cup \{\indshift(k) \mid \text{$\beta_k$ belongs to Case 2}\}.
\end{align*}
Then, the ring $\ring{H}{w_0}=\C[w_0U^-]/\ideal{H}{w_0}$ is isomorphic to the polynomial ring over $\C$ with variables $t_{\ell}$ for $\ell\in\{1,2,\ldots,N\}\setminus D$.
\end{proof}

\vspace{10pt}
\begin{proposition}\label{prop: zero scheme is reduced}
The zero scheme $\mathcal{Z}(x_J,H)$ is reduced, hence an integral scheme.
\end{proposition}

\begin{proof}
We know from Theorem \ref{thm:precup} and Proposition \ref{prop: LCI and CM} that the scheme $\mathcal{Z}(x_J,H)$ is irreducible and Cohen-Macaulay. Also, Lemma \ref{lem:w_0 is polynomial ring} shows that $\mathcal{Z}(x_J,H)$ is generically reduced, i.e.\ it has a reduced point. Since $\mathcal{Z}(x_J,H)$ is Cohen-Macaulay, this shows that $\mathcal{Z}(x_J,H)$ is in fact reduced (\cite[Exercise 18.9]{Eisenbud}).
\end{proof}

\vspace{10pt}
Now Lemma \ref{lem: affine cover for zero} implies the following which can be thought of as a generalization of \cite[Theorem 6]{InskoYong} and \cite[Proposition 3.5]{ab-de-ga-ha}.
\begin{corollary}\label{prop: A is the coordinate ring}
Let $w\in W$ such that $\Hess{x_J}{H}\cap wU^-\neq\emptyset$. Then the ring $\ring{H}{w}$ is an integral domain.
In particular, $\ring{H}{w}$ is the coordinate ring of $\Hess{x_J}{H}\cap wU^-$.
\end{corollary}

\vspace{10pt}
We also record the following property of $\Hess{x_J}{H}$ (cf.\ \cite[Corollary 7]{InskoYong} and \cite[Theorem 4.1]{ab-de-ga-ha}).
\begin{corollary}\label{cor: LCI}
Arbitrary regular Hessenberg variety $\Hess{x_J}{H}$ is a local complete intersection. In particular, it is Cohen-Macaulay.
\end{corollary}
\begin{proof}
We know from Proposition \ref{prop: LCI and CM} that the scheme $\mathcal{Z}(x_J,H)$ is a local complete intersection. Now it is a reduced scheme by Proposition \ref{prop: zero scheme is reduced}. Thus, we conclude that the variety $\Hess{x_J}{H}$ is a local complete intersection.
\end{proof}

\vspace{10pt}
Proposition \ref{prop: zero scheme is reduced} also implies the following formula for the classes of regular Hessenberg varieties in the cohomology group $H^*(G/B)$ by \cite[Proposition 14.1 and Example 14.1.1]{Fulton}.
\begin{corollary}\label{cor: cohomology class}
In $H^*(G/B)$, we have
\begin{align*}
[\Hess{x_J}{H}]
= e(G\times_B(\mathfrak{g}/H)),
\end{align*}
where the right hand side is the Euler class\footnote{In the language of Section 4 (see \eqref{eq:def of L_lambda}), we have $e(G\times_B(\mathfrak{g}/H))=\prod_{\alpha\in \Phi^+ \setminus \HR{H}} c_1(L_{\alpha})$, where $c_1(L_{\alpha})$ is the first Chern class of the line bundle $L_{\alpha}$ over $G/B$.} of the vector bundle $G\times_B(\mathfrak{g}/H)$. In particular, the class of $\Hess{x_J}{H}$ in $H^*(G/B)$ does not depend on $J$. 
\end{corollary}

\bigskip
\section{Line bundles associated to codimension 1 sequences}
Let $\lambda \colon T\rightarrow \C^{\times}$ be a weight of the maximal torus $T$ of $G$. By composing this with the canonical projection $B \twoheadrightarrow T$, we obtain a homomorphism $\lambda \colon B\rightarrow \C^{\times}$ which by slight abuse of notation we also denote by $\lambda$. Let $\C_{\lambda}=\C$ be the 1 dimensional representation of $B$ given by $b\cdot z=\lambda(b)z$ for $b\in B$ and $z\in\C$. We denote by $\C_{\lambda}^*$ its dual representation. Since the quotient map $\pr \colon G\rightarrow G/B$ is a principal $B$-bundle, we can take an associated line bundle 
\begin{align}\label{eq:def of L_lambda}
 \LB_{\lambda} = G\times_B \C_{\lambda}^*.
\end{align}
Namely, it is the quotient of the product $G\times \C$ by the right $B$-action given by $(g,z)\cdot b=(gb,\lambda^{-1}(b^{-1})z)=(gb,\lambda(b)z)$ for $b\in B$ and $(g,z)\in G\times \C$. For a subvariety $X$ of $G/B$, we will also denote by $\LB_{\lambda}$ the restriction of $\LB_{\lambda}$ to $X$ by abusing notation.

Let $H'\subset H\subseteq \mathfrak{g}$ be Hessenberg spaces satisfying condition \eqref{eq:containing simple roots} in Section \ref{sec: regular Hessenberg varieties} such that $\dim_{\C} H' = \dim_{\C} H -1$, equivalently $\HR{H'}=\HR{H} \setminus \{\alpha\}$ for some maximal element $\alpha$ of $\HR{H}$ (with respect to the partial order $\preceq$ on $\Phi^+$ defined in Section \ref{sec: regular Hessenberg varieties}). Then we have a sequence of subvarieties
\begin{align*}
 \Hess{x_J}{H'} \subset \Hess{x_J}{H}\subseteq G/B, 
\end{align*}
where $\Hess{x_J}{H'}$ is of codimension 1 in $\Hess{x_J}{H}$ by Theorem \ref{thm:precup}. Since $H/H'$ is a one-dimensional representation of $B$, we have a line bundle $G\times_B(H/H')$ over $G/B$, and we denote its restriction to $\Hess{x_J}{H}$ by $(G\times_B(H/H'))|_{\Hess{x_J}{H}}$. Then the element $x_J$ induces a well-defined section 
\begin{align*}
 \sectseq \colon \Hess{x_J}{H} \rightarrow (G\times_B(H/H'))|_{\Hess{x_J}{H}}
 \quad ; \quad
 gB \mapsto [g,\overline{\text{Ad}(g^{-1})x_J}],
\end{align*}
and we have 
\begin{align*}
 \Hess{x_J}{H'} = \mathbb{V}(\sectseq).
\end{align*}
Now let $\mathcal{Z}$ be the zero scheme of the section $\sectseq$. This is a closed subscheme of $\Hess{x_J}{H}$ of codimension 1, and we have its ideal sheaf $\mathcal{I}$ on $\Hess{x_J}{H}$ which is an invertible sheaf. By the construction of the zero scheme $\mathcal{Z}$, the line bundle\footnote{Precisely, this is the line bundle $\RSpec(\Sym\mathcal{I}^{\vee})$ over $\Hess{x_J}{H}$, where $\mathcal{I}^{\vee}$ is the dual of $\mathcal{I}$ and $\RSpec$ denotes the relative spectrum over $\Hess{x_J}{H}$.} associated to $\mathcal{I}$ over $\Hess{x_J}{H}$ is isomorphic to the dual of the original line bundle $(G\times_B(H/H'))|_{\Hess{x_J}{H}}$ (\cite[Ch. II, Sect.\ 6]{Hartshorne}). The latter is in fact isomorphic to $L_{-\alpha}(=(G\times_B \C_{\alpha})|_{\Hess{x_J}{H}})$ since
the decompositions
\begin{align*}
H' = \mathfrak{b} \oplus \bigoplus_{\alpha\in \HR{H'}}\mathfrak{g}_{-\alpha}, \quad 
H = \mathfrak{b} \oplus \bigoplus_{\alpha\in \HR{H}}\mathfrak{g}_{-\alpha},
\end{align*}
imply that $(H/H')^* \cong \C_{\alpha}$ as $B$-representations. 

From the descriptions given in Section \ref{sec: defining ideals}, it is clear that $\mathcal{Z}$ is the closed subscheme of $\Hess{x_J}{H}$ whose defining ideal is generated by a single function $\poly{\alpha}{\tilde{w}}$ (see \eqref{eq:def of f_alpha}) in each open set $\Hess{x_J}{H}\cap wU^-$. So, by Lemma \ref{lem: affine cover for zero} and Proposition \ref{prop: zero scheme is reduced}, it follows that $\mathcal{Z}$ is isomorphic to $\mathcal{Z}(x_J,H')$ as closed subschemes of $G/B$. Thus, it follows that $\mathcal{Z}$ is a reduced scheme by Proposition \ref{prop: zero scheme is reduced} again. Combining with the above computations, we now obtain the main claim in this section.

\begin{proposition}\label{prop: ideal sheaf}
Let $H'\subset H$ be Hessenberg spaces such that $\HR{H'}=\HR{H} \setminus \{\alpha\}$ as above and $\mathcal{I}$ the ideal sheaf on $\Hess{x_J}{H}$ of the closed subscheme $\Hess{x_J}{H'}\subset \Hess{x_J}{H}$ of codimension $1$.
Then the line bundle associated to $\mathcal{I}$ is isomorphic to $\LB_{-\alpha}$. 
\end{proposition}

\vspace{10pt}
In the next section, we will use Proposition \ref{prop: ideal sheaf} to study the cohomology groups of the structure sheaves of regular Hessenberg varieties. Here, we give a simple application of Proposition \ref{prop: ideal sheaf} to the classes of regular Hessenberg varieties in the $K$-theory of coherent sheaves on $G/B$. 

Taking $H'\subset H$ as above, we have an exact sequence of $\mathcal{O}_{X}$-modules: 
\begin{align*}
 0 \rightarrow \mathcal{I} \rightarrow \mathcal{O}_{X} \rightarrow \mathcal{O}_{X'} \rightarrow 0,
\end{align*}
where for simplicity we denote by $\mathcal{O}_{X'}$ the direct image of the structure sheaf of $X'$ under the inclusion map $X'\hookrightarrow X$. We have $\mathcal{I}\cong\IS_{-\alpha}$ from Proposition \ref{prop: ideal sheaf}. So, by taking direct images, we obtain an exact sequence of $\mathcal{O}_{G/B}$-modules: 
\begin{align*}
 0 \rightarrow \mathcal{L}_{-\alpha}\otimes_{\mathcal{O}_{G/B}} \mathcal{O}_X \rightarrow \mathcal{O}_{X} \rightarrow \mathcal{O}_{X'} \rightarrow 0.
\end{align*}
Denoting by $K(G/B)$ the $K$-theory of coherent sheaves on $G/B$, this implies an inductive formula
\begin{align*}
 [\mathcal{O}_{X'}] = [\mathcal{O}_{X}] (1-[\mathcal{L}_{-\alpha}]) \quad \text{in $K(G/B)$}.
\end{align*}
Here, the brackets denote equivalence classes in $K(G/B)$.
Now, it is straightforward to prove the following by induction (cf.\ \cite[Theorem 5]{InskoYong}).
\begin{corollary}
Let $X=\Hess{x_J}{H}$ be a regular Hessenberg variety. In $K(G/B)$, we have
\begin{align*}
[\mathcal{O}_{X}]=\prod_{\alpha\in \Phi^+ \setminus \HR{H}} (1-[\mathcal{L}_{-\alpha}]).
\end{align*}
In particular, the class of $\Hess{x_J}{H}$ in $K(G/B)$ does not depend on $J$. 
\end{corollary}

\bigskip
\section{Cohomology groups of structure sheaves}\label{sec:sheaf cohomology}
Let $\Hess{x_J}{H}$ be a regular Hessenberg variety. In this section, we will prove that the higher cohomology groups of the structure sheaf of $\Hess{x_J}{H}$ vanish. As a corollary, we will show that the Picard group $\Pic(\Hess{x_J}{H})$ is isomorphic to the second singular cohomology group $H^2(\Hess{x_J}{H};\Z)$, where we regard $\Hess{x_J}{H}$ as a complex analytic space when we apply the singular cohomology.

\subsection{Borel-Weil-Bott theorem}\label{sec:BWB}
For a weight $\lambda \colon T\rightarrow \C^{\times}$ of the maximal torus $T$ of $G$, we think of $\lambda$ as an element of $\mathfrak{t}^*$ via its differential. 
Let $\rho\coloneqq\frac{1}{2}\sum_{\alpha\in\Phi^+}\alpha=\sum_{i=1}^n\varpi_i$, where $\varpi_1,\ldots,\varpi_n$ are the fundamental weights of $T$, and we write $w\bullet\lambda\coloneqq w(\lambda+\rho)-\rho$ for $w\in W$.
Then, exactly one of the following two holds:
\begin{itemize}
\item[(1)] There is no $w\in W$ such that $w\bullet\lambda$ is dominant, i.e.\ $\lambda+\rho$ lies on the boundary of a Weyl chamber.
\item[(2)] There is a unique $w\in W$ such that $w\bullet\lambda$ is dominant, i.e.\ $\lambda+\rho$ lies in the interior of a Weyl chamber.
\end{itemize}
Note that, in the case (2), we have $\ell(w)=|\{\beta\in\Phi^+ \mid (\beta,\lambda+\rho)<0\}|$. For a dominant weight $\lambda$, let $V_{\lambda}$ be the irreducible representation of $G$ with highest weight $\lambda$. Recall that $\LB_{\lambda}=G\times_B \C_{\lambda}^*$ is the line bundle over $G/B$ associated to $\lambda$ defined in \eqref{eq:def of L_lambda}. Let us denote by $\IS_{\lambda}$ the sheaf of sections of $\LB_{\lambda}$.

The Borel-Weil-Bott theorem (\cite{Jantzen}) now tells us that in the case (1) above, we have
\begin{align*}
 H^i(G/B,\IS_{\lambda}) = 0 \quad (i\in\Z_{\geq0}),
\end{align*}
and in the case (2), we have 
\begin{align*}
 H^i(G/B,\IS_{\lambda}) \cong
 \begin{cases}
  0 \quad &(i\neq \ell(w)), \\
  V_{w\bullet\lambda}^* &(i=\ell(w)).
 \end{cases}
\end{align*}
When $\lambda$ is dominant, the above $w$ is the identity element in $W$, and hence this recovers the Borel-Weil theorem $H^0(G/B,\IS_{\lambda})\cong V_{\lambda}^*$.

\subsection{Higher cohomology groups}
Let $(\ , \ )$ be a $W$-invariant inner product on $\mathfrak{t}^*_{\R}\coloneqq\text{span}_{\mathbb{R}}\Phi\subseteq \mathfrak{t}^*$. 
The following lemma is a corollary of \cite[Lemma 2.13]{Macdonald 3}. 
\begin{lemma}\label{lem: Macdonald}
For any distinct positive roots $\gamma_1,\dots,\gamma_p \ (p\geq1)$ which are not simple roots, there exists a root $\alpha$ such that $(\alpha, \rho-\sum_{i=1}^p\gamma_i)=0$.
\end{lemma}

\begin{proof}
Suppose that there does not exist a root $\alpha$ satisfying $(\alpha, \rho-\sum_{i=1}^p\gamma_i)=0$. Then, by \cite[Lemma 2.13]{Macdonald 3}, we have $\{\gamma_1,\dots,\gamma_p\}=R(w)$ for some $w\in W$, where $R(w)\coloneqq(w\Phi^-)\cap\Phi^+$. Since $\gamma_1,\dots,\gamma_p$ are not simple roots by our assumption, $R(w)$ does not contain any simple roots. Noticing that we can write $R(w) =\{ \alpha\in\Phi^+ \mid w^{-1}\alpha\in\Phi^-\}$, the fact that $R(w)$ does not contain any simple roots implies that $w^{-1}$ sends all simple roots to positive roots. Thus, it follows that $w^{-1}$ is the identity element. But this means that $R(w)=\emptyset$ which contradicts to $\{\gamma_1,\dots,\gamma_p\}\neq\emptyset$.
\end{proof}

\vspace{10pt}
This lemma shows that the weight $-\sum_{i=1}^p\gamma_i$ satisfies the case (1) in the Borel-Weil-Bott theorem. Thus,
for any distinct positive roots $\gamma_1,\dots,\gamma_p\ (p\geq1)$ which are not simple roots, we obtain $H^i(G/B,\IS_{-\gamma})=0$ for all $i$, where $\gamma\coloneqq\sum_{i=1}^p\gamma_i$. In particular, we have
\begin{align}\label{eq:higher vanishing for G/B}
H^i(G/B,\IS_{-\gamma})=0 \quad (i>p).
\end{align}
Note that we also have $H^i(G/B,\mathcal{O}_{G/B})=0$ for $i>0$ by the case (2) of the Borel-Weil-Bott theorem, where $\mathcal{O}_{G/B}(=\IS_0)$ is the structure sheaf of $G/B$. So we see that \eqref{eq:higher vanishing for G/B} holds when $p=0$ as well by regarding $\gamma=0$ in that case.

\begin{theorem}\label{thm: vanishing of higher cohomology}
Let $X=\Hess{x_J}{H}$ be a regular Hessenberg variety. Then
\begin{align*}
H^i(X,\mathcal{O}_X)=0 \quad (i>0),
\end{align*}
where $\mathcal{O}_X$ is the structure sheaf of $X$.
\end{theorem}

\begin{proof}
We will rather prove the following: for any distinct positive roots $\gamma_1,\dots,\gamma_p\in\HR{H} \ (p\geq0)$ which are not simple roots, we have
\begin{align}\label{eq: stronger vanishing}
H^i(X,\IS_{-\gamma})=0 \quad (i>p),
\end{align}
where $\gamma=\sum_{i=1}^p\gamma_i$. Then the desired claim follows from the case $p=0$. We prove this by induction on the dimension of $\mathfrak{g}/H$. We know that \eqref{eq: stronger vanishing} holds when $X=G/B$ (i.e.\ when $\dim\mathfrak{g}/H=0$) as we explained above.

We assume by induction that the claim holds for $X=\Hess{x_J}{H}$, and we show that it holds for $X'=\Hess{x_J}{H'}$ as well, where $H'\subset H$ is a Hessenberg space (satisfying condition \eqref{eq:containing simple roots} in Section \ref{sec: regular Hessenberg varieties}) such that $\HR{H'}=\HR{H}\setminus\{\alpha\}$ for some maximal element $\alpha$ of $\HR{H}$. That is, we assume that $\gamma_1,\ldots,\gamma_p$ $(p\geq 0)$ are distinct positive roots in $\HR{H'}$ which are not simple roots, and we prove $H^i(X',\IS_{-\gamma})=0$ for $i>p$. Let $\mathcal{I}$ be the ideal sheaf of $\Hess{x_J}{H'}$ on $\Hess{x_J}{H}$. Then we have $\mathcal{I}\cong\IS_{-\alpha}$ by Proposition \ref{prop: ideal sheaf}, and thus we have an exact sequence of $\mathcal{O}_{X}$-modules: 
\begin{align*}
 0 \rightarrow \IS_{-\alpha} \rightarrow \mathcal{O}_{X} \rightarrow \mathcal{O}_{X'} \rightarrow 0,
\end{align*}
where for simplicity we denote by $\mathcal{O}_{X'}$ the direct image of the structure sheaf of $X'$ under the inclusion map $X'\hookrightarrow X$.
So, by tensoring $\IS_{-\gamma}$, we obtain an exact sequence
\begin{align*}
 0 \rightarrow \IS_{-\alpha-\gamma} \rightarrow \IS_{-\gamma} \rightarrow \IS_{-\gamma}\otimes_{\mathcal{O}_X}\mathcal{O}_{X'} \rightarrow 0.
\end{align*}
Now we have a long exact sequence of cohomology groups: 
\begin{align*}
 \cdots \rightarrow
 H^i(X,\IS_{-\gamma}) 
 \rightarrow 
 H^i(X',\IS_{-\gamma})
 \rightarrow 
 H^{i+1}(X,\IS_{-\alpha-\gamma}) 
 \rightarrow \cdots.
\end{align*}
Since we have $\alpha\in\HR{H}\setminus\HR{H'}$ and $\gamma_1,\ldots,\gamma_p\in\HR{H'}$, it follows that the elements $\alpha, \gamma_1,\ldots,\gamma_p$ are distinct positive roots in $\HR{H}$ which are not simple roots, so that the inductive hypothesis applies to $H^i(X,\IS_{-\gamma})$ and $H^{i+1}(X,\IS_{-\alpha-\gamma})$ in this sequence. That is, for $i>p$, we have $H^i(X,\IS_{-\gamma})=0$ and $H^{i+1}(X,\IS_{-\alpha-\gamma})=0$, and hence we obtain $H^i(X',\IS_{-\gamma})=0$, as desired. Thus, by induction, we obtain \eqref{eq: stronger vanishing}.
\end{proof}

\vspace{10pt}
Combining this with the GAGA principle \cite{Serre}, we obtain the following corollary. Let $\Pic(\Hess{x_J}{H})$ be the Picard group of $\Hess{x_J}{H}$ and $ H^2(\Hess{x_J}{H};\Z)$ the second singular cohomology group with $\Z$-coefficients, where we regard $\Hess{x_J}{H}$ as a complex analytic space when we apply the singular cohomology.
\begin{corollary}\label{cor: Pic and H^2}
Let $\Hess{x_J}{H}$ be a regular Hessenberg variety. Then we have 
\begin{align*}
\Pic(\Hess{x_J}{H}) \cong H^2(\Hess{x_J}{H};\Z).
\end{align*}
\end{corollary}

\begin{proof}
Let us denote $X=\Hess{x_J}{H}$ for simplicity, and we write $X^{an}$ when we regard $X$ as a complex analytic space\footnote{This is called the analytification of $X$.}. Then \cite{Serre} tells us that the group of isomorphism classes of algebraic line bundles on $X$ and the group of isomorphism classes of analytic line bundles on $X^{an}$ are isomorphic. The former group is exactly the Picard group $\Pic(X)$ which is the same as $H^1(X,\mathcal{O}_{X}^*)$, where $\mathcal{O}_{X}^*$ is the sheaf of invertible regular functions on $X$ \cite[Ch.\ III, Exercise 4.5]{Hartshorne}. Similarly, the latter group is given by $H^1(X^{an},\mathcal{O}_{X^{an}}^*)$, where $\mathcal{O}_{X^{an}}^*$ is the sheaf of invertible holomorphic functions on $X^{an}$.

On $X^{an}$, we have the exponential sequence
\begin{align}\label{eq: exp seq}
 0 \rightarrow \underline{\Z} \rightarrow \mathcal{O}_{X^{an}} \stackrel{\exp}{\rightarrow} \mathcal{O}_{X^{an}}^* \rightarrow 0,
\end{align}
where $\underline{\Z}$ is the constant sheaf on $X$ with values in $\Z$ and $\mathcal{O}_{X^{an}}$ is the sheaf of holomorphic functions on $X^{an}$. Since Theorem \ref{thm: vanishing of higher cohomology} together with \cite{Serre} shows that we have $H^i(X^{an},\mathcal{O}_{X^{an}})=0$ for $i>0$, we obtain an isomorphism $H^1(X^{an},\mathcal{O}_{X^{an}}^*)\cong H^2(X^{an},\underline{\Z})$ from the long exact sequence of cohomology groups arising from \eqref{eq: exp seq}. The sheaf cohomology group $H^2(X^{an},\underline{\Z})$ is isomorphic to the singular cohomology group $H^2(X^{an};\Z)$ with $\Z$-coefficients, and this is exactly the one we write as $H^2(\Hess{x_J}{H};\Z)$ in the claim.
\end{proof}

\vspace{10pt}
For example, when $J=I$, the element $N_0\coloneqq x_I=n_I$ is regular nilpotent, and $\Hess{N_0}{H}$ is a regular nilpotent Hessenberg variety. It is known from \cite{precup13a} and \cite{Insko-Tymoczko} that the restriction map gives us an isomorphism\footnote{In type $A$, this holds with $\Z$-coefficients by \cite{ty} and \cite{Insko}.}
\begin{align*}
H^2(\Hess{N_0}{H};\Q) \cong H^2(G/B;\Q)
\end{align*}
since we are assuming that all the simple roots are contained in $\HR{H}$ (i.e.\ the assumption \eqref{eq:containing simple roots}). Thus, for a regular nilpotent Hessenberg variety $\Hess{N_0}{H}$, we obtain 
\begin{align*}
\Pic(\Hess{N_0}{H})_{\Q} \cong \Pic(G/B)_{\Q},
\end{align*}
where $\Pic(X)_{\Q}\coloneqq\Pic(X)\otimes_{\Z}\Q$.

\bigskip
\section{Family of regular Hessenberg varieties}\label{sec: flat family}
In this section, we study the flat family of regular Hessenberg varieties over the open subvariety $\mathfrak{g}_{{\rm reg}}\subset\mathfrak{g}$ consisting of the regular elements. The goal of this section is to prove that the scheme-theoretic fibers of this family over the closed points are reduced (when we think of this family as the associated integral scheme). As a corollary, we will show that any regular Hessenberg variety admits a flat degeneration to a regular nilpotent Hessenberg variety. In this sense, the regular nilpotent Hessenberg varieties are the most degenerated ones among all regular Hessenberg varieties. We also show that regular Hessenberg varieties associated to a Hessenberg space determine the same cycle in the singular homology $H_*(G/B;\Z)$.

Let $H\subseteq\mathfrak{g}$ be a Hessenberg space satisfying condition \eqref{eq:containing simple roots} in Section \ref{sec: regular Hessenberg varieties}. We begin with the family of all Hessenberg varieties associated to $H$:
\begin{align*}
\AHess{H} \coloneqq \{(gB,x) \in G/B \times \mathfrak{g} \mid \text{Ad}(g^{-1})x \in H \}.
\end{align*}
Observe that we have an isomorphism $\AHess{H} \cong G\times_B H$ given by $(gB,x) \mapsto [g, \text{Ad}(g^{-1})x]$. So, the first projection $\AHess{H}\rightarrow G/B$ is a vector bundle (cf.\ \cite[Sect.\ 8.2]{Brosnan-Chow}). In particular, $\AHess{H}$ is a non-singular variety. Let us denote by $\pi \colon \AHess{H} \rightarrow \mathfrak{g}$ the second projection. 

Let $\mathfrak{g}_{{\rm reg}}$ be the set of regular elements of $\mathfrak{g}$. It is a Zariski open subset of $\mathfrak{g}$ (\cite[Ch.\ 1]{Humphreys95}). Now let $\FHess{H}$ be the family of regular Hessenberg varieties, i.e.\ 
\begin{align}\label{eq:def of family of reg Hess}
\FHess{H} \coloneqq \pi^{-1}(\mathfrak{g}_{{\rm reg}}) = \{(gB,x) \in G/B \times \mathfrak{g}_{{\rm reg}} \mid \text{Ad}(g^{-1})x \in H \}.
\end{align}
Namely, for each $x\in\mathfrak{g}_{{\rm reg}}$, its fiber $\pi^{-1}(x)=\Hess{x}{H}$ is a regular Hessenberg variety. By definition, $\FHess{H}$ has the natural projection $\pi \colon \FHess{H} \rightarrow \mathfrak{g}_{{\rm reg}}$ which by slight abuse of notation we also denote by $\pi$. Note that $\FHess{H}$ is irreducible and non-singular since $\FHess{H}$ is an open subvariety of $\AHess{H}$. In particular, it is Cohen-Macaulay. Also, Theorem \ref{thm:precup} tells us that the fibers of $\pi \colon \FHess{H}\rightarrow \mathfrak{g}_{{\rm reg}}$ have the same dimension. So, by \cite[Sect.\ 23]{Matsumura}, the following holds when we regard the varieties $\FHess{H}$ and $\mathfrak{g}_{{\rm reg}}$ as integral schemes. 

\begin{proposition}\label{prop: flatness}
The morphism \emph{$\pi \colon \FHess{H} \rightarrow \mathfrak{g}_{{\rm reg}}$} is flat.
\end{proposition}

The flatness of the family $\pi \colon \FHess{H} \rightarrow \mathfrak{g}_{{\rm reg}}$ ensures that certain algebro-geometric invariants of the \textit{scheme-theoretic fibers} are constant along the family. However, scheme-theoretic fibers of a family could be non-reduced schemes on some special points, especially on the subset for which the fibers degenerate. In that case, those invariants of a non-reduced fiber could be different from the ones computed with the reduced scheme structure. We will show in the next subsection that the scheme-theoretic fibers of $\pi \colon \FHess{H} \rightarrow \mathfrak{g}_{{\rm reg}}$ over the closed points are reduced.

\subsection{Reducedness of the fibers}\label{subsec: reducedness of fibers}
The goal of this section is to prove the following. Note again that we implicitly think of $\FHess{H}$ and $\mathfrak{g}_{{\rm reg}}$ as their associated integral schemes in the claim. 

\begin{theorem}\label{thm: fibers are reduced}
The scheme-theoretic fibers of \emph{$\pi \colon \FHess{H} \rightarrow \mathfrak{g}_{{\rm reg}}$} over the closed points are reduced.
\end{theorem}

To give a proof, we first consider a scheme structure on $\FHess{H}$ in the same manner as in Section \ref{sec: defining ideals}, and we prove that this scheme is reduced to see that it is exactly the integral scheme appearing in Theorem \ref{thm: fibers are reduced}. 

Recall that in Section \ref{subsec: zero scheme} we considered the vector bundle $G\times_B(\mathfrak{g}/H)$ over $G/B$. This induces a vector bundle 
\begin{align*}
(G\times_B(\mathfrak{g}/H))\times \mathfrak{g}_{{\rm reg}}
\end{align*}
over the product $G/B\times \mathfrak{g}_{{\rm reg}}$. For this vector bundle, we have a well-defined section
\begin{align}\label{eq:family section}
\Fsect \colon G/B\times \mathfrak{g}_{{\rm reg}} \rightarrow (G\times_B(\mathfrak{g}/H))\times \mathfrak{g}_{{\rm reg}}
\quad ; \quad 
(gB,x)\mapsto ([g,\overline{\text{Ad}(g^{-1})x}],x),
\end{align}
where $\overline{\text{Ad}(g^{-1})x}\in\mathfrak{g}/H$ denotes the image of $\text{Ad}(g^{-1})x\in\mathfrak{g}$. It is obvious that $\FHess{H}$ is the zero-set $\mathbb{V}(\Fsect)$.

\begin{definition}
Let $\ZFHess{H}$ be the zero scheme of the section $\Fsect$ given at \eqref{eq:family section}.
\end{definition}

Let $\mathcal{G}_{{\rm reg}}$ be the integral scheme associated to $\mathfrak{g}_{{\rm reg}}$. Then we have a natural morphism $\pi\colon\ZFHess{H}\rightarrow \mathcal{G}_{{\rm reg}}$ corresponding to $\pi \colon \FHess{H} \rightarrow \mathfrak{g}_{{\rm reg}}$ which by abusing notation we also denote by $\pi$. 
Now take a regular element $x\in\mathfrak{g}_{{\rm reg}}$. Recall from Section \ref{sec: background and notation} that we have the zero scheme $\ZHess{x}{H}$ defined by the section $\sect{x} \colon G/B \rightarrow G\times_B(\mathfrak{g}/H)$ given at \eqref{eq: section of the Hessenberg vector bundle}. Observe that this vector bundle $G\times_B(\mathfrak{g}/H)$ and the section $\sect{x}$ are obtained by pulling back the vector bundle and the section $\Fsect$ appearing in \eqref{eq:family section} under the closed embedding $G/B\times\{x\}\hookrightarrow G/B\times \mathfrak{g}_{{\rm reg}}$:
\begin{align*}
\xymatrix{
G\times_B(\mathfrak{g}/H) \ar[r] & (G\times_B(\mathfrak{g}/H))\times \mathfrak{g}_{{\rm reg}} \\
G/B\times\{x\} \ar[u]^{\sect{x}}\ar[r] & G/B\times\mathfrak{g}_{{\rm reg}}  \ar[u]_{\Fsect}.
}
\end{align*}
Thus, it follows that $\ZHess{x}{H}$ is a closed subscheme of $\ZFHess{H}$ given by the base change of $\ZFHess{H}$ under the closed embedding $\Spec k(x)\rightarrow \mathcal{G}_{{\rm reg}}$, where $k(x)$ is the residue field of the closed point $x\in\mathcal{G}_{{\rm reg}}$:
\begin{align}\label{eq: diagram of zero schemes}
\xymatrix{
\ZHess{x}{H} \ar[d]\ar[r] & \ZFHess{H} \ar[d] \\
\Spec k(x) \ar[r] & \mathcal{G}_{{\rm reg}}.
}
\end{align}
Namely, $\ZHess{x}{H}\cong \ZFHess{H}\times_{\mathcal{G}_{{\rm reg}}} \Spec k(x)$ is the scheme-theoretic fiber of $\ZFHess{H}\rightarrow \mathcal{G}_{{\rm reg}}$ at the closed point $x\in\mathcal{G}_{{\rm reg}}$.\\

Before giving a proof of Theorem \ref{thm: fibers are reduced}, let us prove that the zero scheme $\ZFHess{H}$ is reduced which means that it is the integral scheme associated to $\FHess{H}$. Recalling that $\FHess{H}$ is an open subvariety of $\AHess{H}(\cong G\times_B H)$, we have $\dim\ZFHess{H}=\dim \FHess{H}=\dim G\times_B H$. This shows that 
\begin{align*}
\dim(G/B\times\mathfrak{g}_{{\rm reg}}) - \dim\ZFHess{H}
= \dim (\mathfrak{g}/H)=|\Phi^+ \setminus \HR{H}|
\end{align*}
which coincides with the rank of the vector bundle $(G\times_B(\mathfrak{g}/H))\times \mathfrak{g}_{{\rm reg}}$ appearing at \eqref{eq:family section}.
Thus, $\ZFHess{H}$ has the expected codimension, and hence it follows that the scheme $\ZFHess{H}$ is Cohen-Macaulay by the same argument as that to prove Proposition \ref{prop: LCI and CM}. Because of Theorem \ref{thm:precup}, the fibers of $\pi \colon \ZFHess{H}\rightarrow \mathcal{G}_{{\rm reg}}$ over the closed points of $\mathcal{G}_{{\rm reg}}$ have dimension equal to $\dim\ZFHess{H}-\dim\mathcal{G}_{{\rm reg}}(=|\HR{H}|)$. Now, it follows from \cite[Sect.\ 23]{Matsumura} that the morphism $\pi \colon \ZFHess{H}\rightarrow \mathcal{G}_{{\rm reg}}$ is flat.

Let $s\coloneqq x_{\emptyset}\in\mathfrak{g}_{\text{reg}}$ be a regular semisimple element of the form \eqref{eq:Jordan decomp} with $J=\emptyset$. 
We know from above that the scheme-theoretic fiber $\ZFHess{H}\times_{\mathcal{G}_{{\rm reg}}}\Spec k(s)$ is isomorphic to $\ZHess{s}{H}$. So by Proposition \ref{prop: zero scheme is reduced}, it is the integral scheme associated to a regular semisimple Hessenberg variety $\Hess{s}{H}$. As mentioned in Section \ref{sec: background and notation}, it is known from \cite{DMPS} that $\Hess{s}{H}$ is non-singular. So it follows that the scheme-theoretic fiber $\ZFHess{H}\times_{\mathcal{G}_{{\rm reg}}}\Spec k(s)$ is smooth over $k(s)\cong \C$. Since the morphism $\pi \colon \ZFHess{H}\rightarrow \mathcal{G}_{{\rm reg}}$ is flat as we saw above, this shows that the morphism $\pi \colon \ZFHess{H}\rightarrow \mathcal{G}_{{\rm reg}}$ is smooth over an open neighborhood $\mathcal{W}$ of $s$ in $\mathcal{G}_{{\rm reg}}$ \cite[Ch.\ VII, Theorem 1.8]{AK}. Since $\mathcal{W}$ is non-singular, the smoothness of $\pi$ over $\mathcal{W}$ shows that the total space $\pi^{-1}(\mathcal{W})$ is non-singular as well (\cite[Ch.\ VI, Corollary 4.5]{AK} and \cite[Proposition 11.24]{AM}). In particular, it is a reduced scheme. Since it is an open subscheme of $\ZFHess{H}$, this shows that the scheme $\ZFHess{H}$ has some reduced points. Since $\ZFHess{H}$ is irreducible and Cohen-Macaulay, it follows that the scheme $\ZFHess{H}$ is reduced (cf.\ the proof of Proposition \ref{prop: zero scheme is reduced}). Thus, we conclude that $\ZFHess{H}$ is the integral scheme associated to $\FHess{H}$.

\vspace{10pt}
\begin{proof}[Proof of Theorem $\ref{thm: fibers are reduced}$]
Recall that the scheme structure on $\FHess{H}$ in the statement of Theorem \ref{thm: fibers are reduced} is the integral scheme associated to $\FHess{H}$, and we already verified that it is the same as the zero scheme $\ZFHess{H}$. So we rather prove that the scheme-theoretic fibers of $\pi\colon\ZFHess{H}\rightarrow \mathcal{G}_{{\rm reg}}$ over the closed points are reduced.
Let $x\in\mathfrak{g}_{{\rm reg}}$ be a regular element. 
Then there exists $g\in G$ such that $\text{Ad}(g)x=x_J$ for some regular element $x_J$ of the form \eqref{eq:Jordan decomp} with $J\subseteq I$. The adjoint automorphism $\text{Ad}(g) \colon \mathfrak{g}_{{\rm reg}}\rightarrow \mathfrak{g}_{{\rm reg}}$ induces an automorphism $\text{Ad}(g) \colon \mathcal{G}_{{\rm reg}}\rightarrow \mathcal{G}_{{\rm reg}}$ which sends $x$ to $x_J$. We denote by $\overline{\text{Ad}(g)}$ the induced isomorphism from $\Spec k(x)$ to $\Spec k(x_J)$. Now we have a commutative diagram
\begin{align}\label{eq: diagram of fibers}
\xymatrix{
\ZFHess{H}\times_{\mathcal{G}_{{\rm reg}}}\Spec k(x) \ar[d]\ar[r]^{\text{id}\times\overline{\text{Ad}(g)}} & \ZFHess{H}\times_{\mathcal{G}_{{\rm reg}}}\Spec k(x_J) \ar[d] \\
\mathcal{G}_{{\rm reg}} \ar[r]^{\text{Ad}(g)} & \mathcal{G}_{{\rm reg}}, 
}
\end{align}
where both of the horizontal morphisms are isomorphisms. As we already pointed out, the scheme-theoretic fiber $\ZFHess{H}\times_{\mathcal{G}_{{\rm reg}}}\Spec k(x_J)$ is isomorphic to $\ZHess{x_J}{H}$. Thus, by Proposition \ref{prop: zero scheme is reduced}, this is a reduced scheme. Hence, so is $\ZFHess{H}\times_{\mathcal{G}_{{\rm reg}}}\Spec k(x)$ by \eqref{eq: diagram of fibers}.
\end{proof}

\vspace{10pt}
Let $\lambda$ be a regular dominant weight of $T$. Then we have $H^0(G/B,\mathcal{L}_{\lambda})\cong V_{\lambda}^*$ by the Borel-Weil theorem as we saw in Section \ref{sec:BWB}, and hence we obtain the Kodaira embedding $G/B\hookrightarrow \P(V_{\lambda})$. By composing this with the inclusion $\Hess{x}{H}\hookrightarrow G/B$, we obtain a projective embedding of $\Hess{x}{H}\hookrightarrow \P(V_{\lambda})$. Now, Theorem \ref{thm: fibers are reduced} together with \cite[Ch.\ III, Theorem 9.9]{Hartshorne} gives us the following corollary.

\begin{corollary}\label{cor: Hilb poly}
Let $\lambda$ be a regular dominant weight. Then the Hilbert polynomial of the embedding $\Hess{x}{H}\hookrightarrow \P(V_{\lambda})$ does not depend on a choice of $x\in\mathfrak{g}_{{\rm reg}}$. 
\end{corollary}

In Section \ref{sec:volume of Plucker embeddings}, we will compute the volume of this embedding of $\Hess{x}{H}$ (i.e.\ the leading coefficient of its Hilbert polynomial) by this corollary.

\smallskip

\subsection{Degenerations to regular nilpotent Hessenberg varieties}\label{subsec:degenerations}
As another corollary of Theorem \ref{thm: fibers are reduced}, we show that any regular Hessenberg variety admits a flat degeneration to a regular nilpotent Hessenberg variety. In this sense, the regular nilpotent Hessenberg varieties are the most degenerated ones among all regular Hessenberg varieties. We say that a scheme $X$ \textit{admits a flat degeneration to a scheme $X_0$} if there exists a flat morphism $\mathfrak{X}\rightarrow\Spec\C[t]$ of schemes such that the scheme-theoretic fiber over the closed point $t\neq0$ is isomorphic to $X$, and the scheme-theoretic fiber over the closed point $0$ is $X_0$. 

Let $x_J\in\mathfrak{g}_{{\rm reg}}$ be a regular element of the form \eqref{eq:Jordan decomp} and $N_0\coloneqq \sum_{i\in I}E_{\alpha_i}$ the regular nilpotent element of the form \eqref{eq:Jordan decomp}. We can take a line in $\mathfrak{g}$ given by
\begin{align}\label{eq:degeneration line}
x(t)\coloneqq tx_J+(1-t)N_0 \quad (t\in\C)
\end{align}
connecting $x_J$ and $N_0$. One can prove that, for any fixed $t\neq0$, the element $x(t)$ is conjugate to $tx_J$ (see Appendix for a proof). In particular, this is a line in $\mathfrak{g}_{{\rm reg}}$.

Recall from Section \ref{subsec: reducedness of fibers} that $\pi \colon \ZFHess{H}\rightarrow\mathcal{G}_{{\rm reg}}$ is the morphism of integral schemes corresponding to the projection $\pi\colon\FHess{H}\rightarrow \mathfrak{g}_{{\rm reg}}$. By abusing notation, let us denote by $\Spec\C[t]$ the integral scheme associated to the line given by \eqref{eq:degeneration line}. Then the set of closed points of the scheme $\ZFHessline{H}\coloneqq \ZFHess{H}\times_{\mathcal{G}_{{\rm reg}}}\Spec\C[t]$ describes the preimage of the line given by \eqref{eq:degeneration line} under the projection $\FHess{H}\rightarrow \mathfrak{g}_{{\rm reg}}$. We have a natural morphism $\pi' \colon \ZFHessline{H}\rightarrow\Spec\C[t]$, and it follows from Proposition \ref {prop: flatness} that $\pi'$ is flat since flatness is stable under base change. By Theorem \ref{thm: fibers are reduced}, the scheme-theoretic fibers of $\pi'$ over the closed points are reduced as well since scheme-theoretic fibers are preserved under base change. Namely, the scheme-theoretic fiber over $t\neq0$ is isomorphic to the integral scheme associated to $\Hess{x_J}{H}$, and the scheme-theoretic fiber over $0$ is the integral scheme associated to $\Hess{N_0}{H}$. Thus, we obtain a flat degeneration of $\Hess{x_J}{H}$ to $\Hess{N_0}{H}$. 

\begin{corollary}\label{coro: flat degenerations}
Any regular Hessenberg variety admits a flat degeneration to a regular nilpotent Hessenberg variety.
\end{corollary}

For example, by taking a regular semisimple element $s\in\mathfrak{g}_{\text{reg}}$ and the Hessenberg space $H$ for which $\HR{H}=\Delta$, we see that the toric variety $\Hess{s}{H}$ (see Section \ref{sec: regular Hessenberg varieties}) admits a flat degeneration to the Peterson variety (cf.\ \cite[Sect.\ 4.1]{Kostant}). 

\begin{remark}
In Lie type $A$, some related works have appeared in \cite{an-ty} and \cite{ab-de-ga-ha}.	
Anderson-Tymoczko considered in \cite{an-ty} Hessenberg varieties with the scheme structures, which are isomorphic to ours,  described in terms of degeneracy loci of morphisms of vector bundles $($\cite[Theorem 7.6]{an-ty}$)$. Also,  special case $($when $x_J$ is regular semisimple$)$ of Corollary \ref{coro: flat degenerations} was proved in \emph{\cite[Theorem 5.1]{ab-de-ga-ha}}.
\end{remark}

%

\bigskip
\section{Volumes of projective embeddings}\label{sec:volume of Plucker embeddings}
For a regular dominant weight $\lambda$ of $T$, we have the Kodaira embedding $G/B\hookrightarrow \P(V_{\lambda})$ for the line bundle $L_{\lambda}$.
Let $\Hess{x}{H}$ be a regular Hessenberg variety. By composing this with the inclusion $\Hess{x}{H}\hookrightarrow G/B$, we obtain a closed embedding $\Hess{x}{H}\hookrightarrow \P(V_{\lambda})$. We define the volume of $\LB_{\lambda}$ on $\Hess{x}{H}$ by
\begin{align*}
\text{Vol} (\Hess{x}{H},\LB_{\lambda}) \coloneqq  \frac{1}{d!} \deg (\Hess{x}{H}\hookrightarrow \P(V_{\lambda})),
\end{align*}
where $d\coloneqq \dim_{\C}\Hess{x}{H}=|\HR{H}|$. In this section, we give a Lie theoretic description of the volume $\text{Vol} (\Hess{x}{H},\LB_{\lambda})$. Specializing to the case for a regular nilpotent Hessenberg variety, we will also explain a connection of our volume formula to its cohomology ring.

Let $\mathfrak{t}_{\R}^*\coloneqq \text{span}_{\R}\{\alpha\mid \alpha\in\Phi\}\subset \mathfrak{t}^*$ and $R=\Sym(\mathfrak{t}_{\R}^*)$ the symmetric algebra of $\mathfrak{t}_{\R}^*$. Take a $W$-invariant inner product $(\ , \ )$ on $\mathfrak{t}_{\R}^*$. Since this induces a linear isomorphism $\mathfrak{t}_{\R}^*\cong\mathfrak{t}_{\R}$, we can regard $\mathfrak{t}_{\R}^*$ as the linear functions on itself and $R$ as the ring of polynomial functions on $\mathfrak{t}_{\R}^*$. For example, we have $(fg)(\mu)=f(\mu)g(\mu)=(f,\mu)(g,\mu)$ for $f, g\in \mathfrak{t}_{\R}^*$ and $\mu\in \mathfrak{t}_{\R}^*$. Also, for an element $f\in\mathfrak{t}_{\R}^*$, we define a derivation $\partial_{f}$ on $R$ determined by the Leibniz rule and the linear formula $\partial_{f} = (f, \ \cdot \ )$ on $\mathfrak{t}_{\R}^*$. 

Let us introduce the following polynomial in $R$ which appeared in \cite[Sect.\ 11]{AHMMS}\footnote{Strictly speaking, the polynomial $P_{H}$ without the denominator $Z$ is the one introduced in \cite{AHMMS}.}.
\begin{align}\label{eq: def of volume formula}
P_{H} \coloneqq  \frac{1}{Z}\left(\prod_{\alpha\in \Phi^+ \setminus \HR{H}}\partial_{\alpha}\right) \prod_{\alpha\in\Phi^+} \alpha 
\quad \in R,
\end{align}
where $Z\coloneqq \prod_{\alpha\in\Phi^+}(\alpha,\rho)\in\R^{\times}$. Note that $P_{H}$ is a polynomial function on $\mathfrak{t}_{\R}^*$ as we explained above.
For example, when $H=\mathfrak{g}$ (i.e.\ when $\Hess{x}{H}=G/B$), we have $P_{\mathfrak{g}}(\mu) = \prod_{\alpha\in\Phi^+} \frac{(\alpha,\mu)}{(\alpha,\rho)}$ for $\mu\in\mathfrak{t}_{\R}^*$.

\begin{proposition}\label{prop: volume formula}
Let $\Hess{x}{H}$ be a regular Hessenberg variety and $\lambda$ a regular dominant weight of $T$. Then
\begin{align*}
\emph{Vol} (\Hess{x}{H},\LB_{\lambda}) = P_{H}(\lambda).
\end{align*}
\end{proposition}

\begin{remark}
The $P_{H}$ as an element of $R$ depends on the choice of the $W$-invariant inner product on $\mathfrak{t}_{\R}^*$, but $P_{H}$ as a function on $\mathfrak{t}_{\R}^*$ does not.
\end{remark}

\begin{proof}[Proof of Proposition $\ref{prop: volume formula}$]
By Corollary \ref{cor: Hilb poly}, the degree of the projective embedding $\Hess{x}{H}\hookrightarrow \P(V_{\lambda})$ does not depend on a choice of a regular element $x$. So it suffices to compute the volume of this embedding of a regular semisimple Hessenberg variety.

Let $s\in\mathfrak{t}$ be a regular semisimple element of $\mathfrak{g}$. Since $\Hess{s}{H}$ is a non-singular projective variety which is preserved by the $T$-action on $G/B$, the degree of a closed embedding into a projective space is exactly its symplectic volume, and it can be computed by the Atiyah-Bott-Berline-Vergne (\cite{at-bo} and \cite{be-ve}) localization formula for integrations. So let us write down this formula in what follows. Recall that the line bundle over $G/B$ corresponding to the embedding $G/B\hookrightarrow \P(V_{\lambda})$ is $\LB_{\lambda}$. So, by the localization formula, we have
\begin{align}\label{eq:degree}
\text{Vol} (\Hess{s}{H},\LB_{\lambda})
= \frac{1}{d!}\int_{\Hess{s}{H}} c_1(L_{\lambda})^d 
= \frac{1}{d!}\sum_{w\in W} \frac{(-w\lambda)^d}{\prod_{\alpha\in \HR{H}}(-w\alpha)}.
\end{align}
Here, in the RHS, we think of $w\in W$ as a $T$-fixed point $wB\in G/B$, and the numerator and the denominator are the weights of the $T$-actions on the fiber $\LB_{\lambda}|_{wB}$ and the tangent space $T_{wB} \Hess{s}{H}$ respectively (see \cite{DMPS} for the description of $T_{wB} \Hess{s}{H}$), where we think of each fraction as an element of the quotient field $K$ of the symmetric algebra $R=\Sym(\mathfrak{t}_{\R}^*)$. Now we think of the RHS of \eqref{eq:degree} as a polynomial in the coefficients $\lambda_1,\ldots,\lambda_n$ of $\lambda$ with respect to the fundamental weights $\varpi_1,\ldots,\varpi_n$. More precisely, let
\begin{align*}
P_{\Hess{s}{H}}(y_1,\dots,y_n)\coloneqq \frac{1}{d!}\sum_{w\in W} \frac{(-wy)^d}{\prod_{\alpha\in \HR{H}}(-w\alpha)} \in K[y_1,\dots,y_n],
\end{align*}
where $y\coloneqq \sum_{i=1}^n y_i\varpi_i$. For any $(m_1,\ldots,m_n)\in\Z^n$, we know from \cite{at-bo} and \cite{be-ve} that $P_{\Hess{s}{H}}(m_1,\dots,m_n)$ is the value of the integration $\frac{1}{d!}\int_{\Hess{s}{H}}c_1(L_{\sum_{i=1} ^n m_i \varpi_i})^d \in\R$. From this, it is not difficult to verify that 
$P_{\Hess{s}{H}}(r_1,\dots,r_n)\in\R$  for all $(r_1,\ldots,r_n)\in\R^n$ by continuity. This shows that we in fact have
\begin{align*}
P_{\Hess{s}{H}}(y_1,\dots,y_n) \in \R[y_1,\dots,y_n].
\end{align*}
By construction, we have
\begin{align}\label{eq: volume first expression}
\text{Vol} (\Hess{s}{H},\LB_{\lambda}) = P_{\Hess{s}{H}}(\lambda_1,\dots,\lambda_n)
\end{align}
for any regular dominant weight $\lambda=\sum_{i=1}^{n}\lambda_i\varpi_i$.

Now we associate to each root $\alpha$ a derivation on $K[y_1,\dots,y_n]$ over $K$ by
\begin{align*}
\partial'_{\alpha} \coloneqq  \sum_{k=1}^n(\alpha,\alpha_k^{\vee})\frac{\partial}{\partial y_k},
\end{align*}
where $\alpha_k^{\vee}=\frac{2}{(\alpha_k,\alpha_k)}\alpha_k$. Then we have
$\partial'_{\alpha}(wy) = w\alpha$
for $w\in W$, where $y=\sum_{i=1}^n y_i\varpi_i$ as above. Now it is straightforward to see 
\begin{align}\label{eq: volume using G/B}
P_{\Hess{s}{H}}(y_1,\dots,y_n) = \left(\prod_{\alpha\in \Phi^+ \setminus \HR{H}}\partial'_{\alpha}\right) P_{G/B}(y_1,\dots,y_n).
\end{align}
Here, $P_{G/B}(y_1,\dots,y_n)$ describes the volume formula of the embedding $G/B\hookrightarrow \P(V_{\lambda})$, and it is well-known that 
\begin{align*}
\text{Vol} (G/B,\LB_{\lambda}) 
= \prod_{\alpha\in\Phi^+} \frac{(\alpha,\lambda)}{(\alpha,\rho)}
= \frac{1}{Z} \prod_{\alpha\in\Phi^+} (\alpha,\lambda)
\end{align*}
for any regular dominant weight $\lambda=\sum_{i=1}^{n} \lambda_i\varpi_i$. So,  it follows that
\begin{align}\label{eq:poly exp of P G/B}
P_{G/B}(y_1,\dots,y_n)
= \frac{1}{Z} \prod_{\alpha\in\Phi^+} (\alpha,y)
\end{align}
since both sides are polynomials in $y_1,\ldots,y_n$ over $\R$. Recalling that we regard $R=\Sym(\mathfrak{t}_{\R}^*)$ as the ring of polynomial functions on $\mathfrak{t}_{\R}^*$ through the inner product $(\ , \ )$ (see the explanation above \eqref{eq: def of volume formula} for details), we can express \eqref{eq:poly exp of P G/B} as
\begin{align*}
P_{G/B}
= \frac{1}{Z} \prod_{\alpha\in\Phi^+} \alpha \ \ \in R
\end{align*}
under the ring isomorphism $\psi\colon R\stackrel{\sim}{\rightarrow}\R[y_1,\dots,y_n]$ given by $\psi(\alpha_i^{\vee})=y_i$ for $i=1,\ldots,n$. 
It is straightforward to verify that this isomorphism $\psi$ fits into the following commutative diagram.
\begin{align*}
\xymatrix{
R \ar[d]_{\partial_{\alpha}}\ar[r]^{\hspace{-25pt}\psi} & \R[y_1,\dots,y_n] \ar[d]^{\partial'_{\alpha}} \\
R \ar[r]^{\hspace{-25pt}\psi} & \R[y_1,\dots,y_n]
}
\end{align*}
Namely, the derivation $\partial'_{\alpha}$ on $\R[y_1,\dots,y_n]$ is identical to the derivation $\partial_{\alpha}$ on $R$ defined above \eqref{eq: def of volume formula} (i.e.\ the one which satisfies $\partial_{\alpha} = (\alpha, \ \cdot \ )$ on $\mathfrak{t}_{\R}^*$) under the isomorphism $\psi$.
Thus, we obtain from \eqref{eq: volume using G/B} that
\begin{align}\label{eq: volume using G/B 2}
P_{\Hess{s}{H}} = \left(\prod_{\alpha\in \Phi^+ \setminus \HR{H}}\partial_{\alpha}\right) P_{G/B} = \frac{1}{Z} \left(\prod_{\alpha\in \Phi^+ \setminus \HR{H}}\partial_{\alpha}\right) \prod_{\alpha\in\Phi^+} \alpha
\end{align}
as elements of $R$, and this is exactly $P_H$ given in \eqref{eq: def of volume formula}. Hence, by \eqref{eq: volume first expression}, we obtain
\begin{align*}
\text{Vol} (\Hess{s}{H},\LB_{\lambda}) = P_{H}(\lambda)
\end{align*}
as desired.
\end{proof}

\begin{example}
\emph{
Let $G=\text{SL}_3(\C)$. Take a regular semisimple element $s$ in $\mathfrak{sl}_3(\C)$ and a Hessenberg space $H\subseteq \mathfrak{sl}_3(\C)$ such that $\Phi^+_H=\Delta$. Then $X\coloneqq \Hess{s}{H}$ is a $2$-dimensional toric variety which is called a permutohedral variety (see Section \ref{sec: background and notation}). More precisely, let $v_1,\ldots,v_6$ be elements of $N=\Z^2$ given by $v_1=(1,0)$, $v_2=(1,1)$, $v_3=(0,1)$, $v_4=(-1,0)$, $v_5=(-1,-1)$, $v_6=(0,-1)$. Consider the fan in $N_{\R}=N\otimes_{\Z}\R=\R^2$ whose maximal cones are $\sigma_{12}, \sigma_{23}, \sigma_{34}, \sigma_{45}, \sigma_{56}, \sigma_{61}$, where $\sigma_{ij}$ is the cone spanned by $v_i$ and $v_j$. Then $X$ is the toric variety associated with this fan (see Figure \ref{pic:fan and polytope}).
}

\emph{
Set $\lambda=\rho=\varpi_1+\varpi_2$, where $\varpi_1, \varpi_2$ are the fundamental weights of $T$, and we compute the volume $\text{Vol}(X,L_\rho)$ in two ways. Let us first compute it from the theory of toric varieties (see \cite[Sects.\ 6.1 and 13.4]{co-li-sc-11} for details). We have $L_\rho|_X=K_X^{\vee}$, where $K_X^{\vee}$ is the anti-canonical line bundle. The divisor corresponding to $K_X^{\vee}$ is linearly equivalent to the divisor $D=\sum_{i=1}^6 D_i$, where $D_i$ is the torus invariant divisor on $X$ corresponding to $v_i.$ This determines a polyhedral $P_D \coloneqq \{u\in M_\R \mid \langle u,v_i\rangle\geq -1\ (1\leq i\leq 6)\}$, where $M=\Hom(N, \Z)$ and $M_{\R}=M\otimes_\Z\R$. It is easy to verify that $P_D$ in fact is a polytope in this case (see Figure \ref{pic:fan and polytope}).
\begin{figure}[htbp]
\[
{\unitlength 0.1in%
\begin{picture}(50.0000,20.0000)(4.0000,-26.0000)%
%
\special{pn 8}%
\special{pa 1600 2600}%
\special{pa 1600 600}%
\special{fp}%
\special{sh 1}%
\special{pa 1600 600}%
\special{pa 1580 667}%
\special{pa 1600 653}%
\special{pa 1620 667}%
\special{pa 1600 600}%
\special{fp}%
%
\special{pn 8}%
\special{pa 600 1600}%
\special{pa 2600 1600}%
\special{fp}%
\special{sh 1}%
\special{pa 2600 1600}%
\special{pa 2533 1580}%
\special{pa 2547 1600}%
\special{pa 2533 1620}%
\special{pa 2600 1600}%
\special{fp}%
%
\special{pn 8}%
\special{pa 600 2600}%
\special{pa 2600 600}%
\special{dt 0.045}%
\put(19.1000,-17.6000){\makebox(0,0)[lb]{$v_1$}}%
\put(14.2000,-12.9000){\makebox(0,0)[lb]{$v_3$}}%
\put(20.0000,-13.2000){\makebox(0,0)[lb]{$v_2$}}%
\put(11.1000,-17.2000){\makebox(0,0)[lb]{$v_4$}}%
\put(16.2000,-21.2000){\makebox(0,0)[lb]{$v_6$}}%
\put(11.8000,-21.3000){\makebox(0,0)[lb]{$v_5$}}%
\put(4.0000,-8.0000){\makebox(0,0)[lb]{$N_{\R}$}}%
%
\special{pn 8}%
\special{pa 630 670}%
\special{pa 630 870}%
\special{fp}%
\special{pa 630 870}%
\special{pa 430 870}%
\special{fp}%
%
\special{pn 8}%
\special{pa 1600 1600}%
\special{pa 2000 1600}%
\special{fp}%
\special{sh 1}%
\special{pa 2000 1600}%
\special{pa 1933 1580}%
\special{pa 1947 1600}%
\special{pa 1933 1620}%
\special{pa 2000 1600}%
\special{fp}%
\special{pa 1600 1600}%
\special{pa 1600 1200}%
\special{fp}%
\special{sh 1}%
\special{pa 1600 1200}%
\special{pa 1580 1267}%
\special{pa 1600 1253}%
\special{pa 1620 1267}%
\special{pa 1600 1200}%
\special{fp}%
\special{pa 1600 1600}%
\special{pa 1200 1600}%
\special{fp}%
\special{sh 1}%
\special{pa 1200 1600}%
\special{pa 1267 1620}%
\special{pa 1253 1600}%
\special{pa 1267 1580}%
\special{pa 1200 1600}%
\special{fp}%
\special{pa 1600 1600}%
\special{pa 1600 2000}%
\special{fp}%
\special{sh 1}%
\special{pa 1600 2000}%
\special{pa 1620 1933}%
\special{pa 1600 1947}%
\special{pa 1580 1933}%
\special{pa 1600 2000}%
\special{fp}%
\special{pa 1600 1600}%
\special{pa 2000 1200}%
\special{fp}%
\special{sh 1}%
\special{pa 2000 1200}%
\special{pa 1939 1233}%
\special{pa 1962 1238}%
\special{pa 1967 1261}%
\special{pa 2000 1200}%
\special{fp}%
\special{pa 1600 1600}%
\special{pa 1200 2000}%
\special{fp}%
\special{sh 1}%
\special{pa 1200 2000}%
\special{pa 1261 1967}%
\special{pa 1238 1962}%
\special{pa 1233 1939}%
\special{pa 1200 2000}%
\special{fp}%
\put(31.7000,-8.0000){\makebox(0,0)[lb]{$M_{\R}$}}%
%
\special{pn 8}%
\special{pa 3430 670}%
\special{pa 3430 870}%
\special{fp}%
\special{pa 3430 870}%
\special{pa 3230 870}%
\special{fp}%
%
\special{pn 0}%
\special{sh 0.300}%
\special{pa 4000 1200}%
\special{pa 4400 1200}%
\special{pa 4800 1600}%
\special{pa 4800 2000}%
\special{pa 4400 2000}%
\special{pa 4000 1600}%
\special{pa 4000 1200}%
\special{ip}%
\special{pn 8}%
\special{pa 4000 1200}%
\special{pa 4400 1200}%
\special{pa 4800 1600}%
\special{pa 4800 2000}%
\special{pa 4400 2000}%
\special{pa 4000 1600}%
\special{pa 4000 1200}%
\special{pa 4400 1200}%
\special{fp}%
%
\special{pn 8}%
\special{pa 3400 1600}%
\special{pa 5400 1600}%
\special{fp}%
\special{sh 1}%
\special{pa 5400 1600}%
\special{pa 5333 1580}%
\special{pa 5347 1600}%
\special{pa 5333 1620}%
\special{pa 5400 1600}%
\special{fp}%
%
\special{pn 8}%
\special{pa 4400 2600}%
\special{pa 4400 600}%
\special{fp}%
\special{sh 1}%
\special{pa 4400 600}%
\special{pa 4380 667}%
\special{pa 4400 653}%
\special{pa 4420 667}%
\special{pa 4400 600}%
\special{fp}%
%
\special{sh 1.000}%
\special{ia 4400 1600 20 20 0.0000000 6.2831853}%
\special{pn 8}%
\special{ia 4400 1600 20 20 0.0000000 6.2831853}%
%
\special{sh 1.000}%
\special{ia 4800 1600 20 20 0.0000000 6.2831853}%
\special{pn 8}%
\special{ia 4800 1600 20 20 0.0000000 6.2831853}%
%
\special{sh 1.000}%
\special{ia 4800 2000 20 20 0.0000000 6.2831853}%
\special{pn 8}%
\special{ia 4800 2000 20 20 0.0000000 6.2831853}%
%
\special{sh 1.000}%
\special{ia 4400 2000 20 20 0.0000000 6.2831853}%
\special{pn 8}%
\special{ia 4400 2000 20 20 0.0000000 6.2831853}%
%
\special{sh 1.000}%
\special{ia 4000 2000 20 20 0.0000000 6.2831853}%
\special{pn 8}%
\special{ia 4000 2000 20 20 0.0000000 6.2831853}%
%
\special{sh 1.000}%
\special{ia 4000 1600 20 20 0.0000000 6.2831853}%
\special{pn 8}%
\special{ia 4000 1600 20 20 0.0000000 6.2831853}%
%
\special{sh 1.000}%
\special{ia 4000 1200 20 20 0.0000000 6.2831853}%
\special{pn 8}%
\special{ia 4000 1200 20 20 0.0000000 6.2831853}%
%
\special{sh 1.000}%
\special{ia 4400 1200 20 20 0.0000000 6.2831853}%
\special{pn 8}%
\special{ia 4400 1200 20 20 0.0000000 6.2831853}%
%
\special{sh 1.000}%
\special{ia 4800 1200 20 20 0.0000000 6.2831853}%
\special{pn 8}%
\special{ia 4800 1200 20 20 0.0000000 6.2831853}%
%
\special{sh 1.000}%
\special{ia 4800 800 20 20 0.0000000 6.2831853}%
\special{pn 8}%
\special{ia 4800 800 20 20 0.0000000 6.2831853}%
%
\special{sh 1.000}%
\special{ia 5200 800 20 20 0.0000000 6.2831853}%
\special{pn 8}%
\special{ia 5200 800 20 20 0.0000000 6.2831853}%
%
\special{sh 1.000}%
\special{ia 5200 1200 20 20 0.0000000 6.2831853}%
\special{pn 8}%
\special{ia 5200 1200 20 20 0.0000000 6.2831853}%
%
\special{sh 1.000}%
\special{ia 5200 2000 20 20 0.0000000 6.2831853}%
\special{pn 8}%
\special{ia 5200 2000 20 20 0.0000000 6.2831853}%
%
\special{sh 1.000}%
\special{ia 5200 2400 20 20 0.0000000 6.2831853}%
\special{pn 8}%
\special{ia 5200 2400 20 20 0.0000000 6.2831853}%
%
\special{sh 1.000}%
\special{ia 4800 2400 20 20 0.0000000 6.2831853}%
\special{pn 8}%
\special{ia 4800 2400 20 20 0.0000000 6.2831853}%
%
\special{sh 1.000}%
\special{ia 4400 2400 20 20 0.0000000 6.2831853}%
\special{pn 8}%
\special{ia 4400 2400 20 20 0.0000000 6.2831853}%
%
\special{sh 1.000}%
\special{ia 4000 2400 20 20 0.0000000 6.2831853}%
\special{pn 8}%
\special{ia 4000 2400 20 20 0.0000000 6.2831853}%
%
\special{sh 1.000}%
\special{ia 3600 2400 20 20 0.0000000 6.2831853}%
\special{pn 8}%
\special{ia 3600 2400 20 20 0.0000000 6.2831853}%
%
\special{sh 1.000}%
\special{ia 3600 2000 20 20 0.0000000 6.2831853}%
\special{pn 8}%
\special{ia 3600 2000 20 20 0.0000000 6.2831853}%
%
\special{sh 1.000}%
\special{ia 3600 1600 20 20 0.0000000 6.2831853}%
\special{pn 8}%
\special{ia 3600 1600 20 20 0.0000000 6.2831853}%
%
\special{sh 1.000}%
\special{ia 3600 1200 20 20 0.0000000 6.2831853}%
\special{pn 8}%
\special{ia 3600 1200 20 20 0.0000000 6.2831853}%
%
\special{sh 1.000}%
\special{ia 4000 800 20 20 0.0000000 6.2831853}%
\special{pn 8}%
\special{ia 4000 800 20 20 0.0000000 6.2831853}%
%
\special{sh 1.000}%
\special{ia 3600 800 20 20 0.0000000 6.2831853}%
\special{pn 8}%
\special{ia 3600 800 20 20 0.0000000 6.2831853}%
%
\special{sh 1.000}%
\special{ia 5200 1600 20 20 0.0000000 6.2831853}%
\special{pn 8}%
\special{ia 5200 1600 20 20 0.0000000 6.2831853}%
%
\special{sh 1.000}%
\special{ia 4400 800 20 20 0.0000000 6.2831853}%
\special{pn 8}%
\special{ia 4400 800 20 20 0.0000000 6.2831853}%
%
\special{sh 1.000}%
\special{ia 2400 2000 20 20 0.0000000 6.2831853}%
\special{pn 8}%
\special{ia 2400 2000 20 20 0.0000000 6.2831853}%
%
\special{sh 1.000}%
\special{ia 2400 2400 20 20 0.0000000 6.2831853}%
\special{pn 8}%
\special{ia 2400 2400 20 20 0.0000000 6.2831853}%
%
\special{sh 1.000}%
\special{ia 2000 2000 20 20 0.0000000 6.2831853}%
\special{pn 8}%
\special{ia 2000 2000 20 20 0.0000000 6.2831853}%
%
\special{sh 1.000}%
\special{ia 2000 2400 20 20 0.0000000 6.2831853}%
\special{pn 8}%
\special{ia 2000 2400 20 20 0.0000000 6.2831853}%
%
\special{sh 1.000}%
\special{ia 1600 2000 20 20 0.0000000 6.2831853}%
\special{pn 8}%
\special{ia 1600 2000 20 20 0.0000000 6.2831853}%
%
\special{sh 1.000}%
\special{ia 1600 2400 20 20 0.0000000 6.2831853}%
\special{pn 8}%
\special{ia 1600 2400 20 20 0.0000000 6.2831853}%
%
\special{sh 1.000}%
\special{ia 1200 2000 20 20 0.0000000 6.2831853}%
\special{pn 8}%
\special{ia 1200 2000 20 20 0.0000000 6.2831853}%
%
\special{sh 1.000}%
\special{ia 1200 2400 20 20 0.0000000 6.2831853}%
\special{pn 8}%
\special{ia 1200 2400 20 20 0.0000000 6.2831853}%
%
\special{sh 1.000}%
\special{ia 800 2000 20 20 0.0000000 6.2831853}%
\special{pn 8}%
\special{ia 800 2000 20 20 0.0000000 6.2831853}%
%
\special{sh 1.000}%
\special{ia 800 2400 20 20 0.0000000 6.2831853}%
\special{pn 8}%
\special{ia 800 2400 20 20 0.0000000 6.2831853}%
%
\special{sh 1.000}%
\special{ia 800 1200 20 20 0.0000000 6.2831853}%
\special{pn 8}%
\special{ia 800 1200 20 20 0.0000000 6.2831853}%
%
\special{sh 1.000}%
\special{ia 800 1600 20 20 0.0000000 6.2831853}%
\special{pn 8}%
\special{ia 800 1600 20 20 0.0000000 6.2831853}%
%
\special{sh 1.000}%
\special{ia 1200 1200 20 20 0.0000000 6.2831853}%
\special{pn 8}%
\special{ia 1200 1200 20 20 0.0000000 6.2831853}%
%
\special{sh 1.000}%
\special{ia 1200 1600 20 20 0.0000000 6.2831853}%
\special{pn 8}%
\special{ia 1200 1600 20 20 0.0000000 6.2831853}%
%
\special{sh 1.000}%
\special{ia 1600 1200 20 20 0.0000000 6.2831853}%
\special{pn 8}%
\special{ia 1600 1200 20 20 0.0000000 6.2831853}%
%
\special{sh 1.000}%
\special{ia 1600 1600 20 20 0.0000000 6.2831853}%
\special{pn 8}%
\special{ia 1600 1600 20 20 0.0000000 6.2831853}%
%
\special{sh 1.000}%
\special{ia 2000 1200 20 20 0.0000000 6.2831853}%
\special{pn 8}%
\special{ia 2000 1200 20 20 0.0000000 6.2831853}%
%
\special{sh 1.000}%
\special{ia 2000 1600 20 20 0.0000000 6.2831853}%
\special{pn 8}%
\special{ia 2000 1600 20 20 0.0000000 6.2831853}%
%
\special{sh 1.000}%
\special{ia 2400 1200 20 20 0.0000000 6.2831853}%
\special{pn 8}%
\special{ia 2400 1200 20 20 0.0000000 6.2831853}%
%
\special{sh 1.000}%
\special{ia 2400 1600 20 20 0.0000000 6.2831853}%
\special{pn 8}%
\special{ia 2400 1600 20 20 0.0000000 6.2831853}%
%
\special{sh 1.000}%
\special{ia 2400 800 20 20 0.0000000 6.2831853}%
\special{pn 8}%
\special{ia 2400 800 20 20 0.0000000 6.2831853}%
%
\special{sh 1.000}%
\special{ia 2000 800 20 20 0.0000000 6.2831853}%
\special{pn 8}%
\special{ia 2000 800 20 20 0.0000000 6.2831853}%
%
\special{sh 1.000}%
\special{ia 1600 800 20 20 0.0000000 6.2831853}%
\special{pn 8}%
\special{ia 1600 800 20 20 0.0000000 6.2831853}%
%
\special{sh 1.000}%
\special{ia 1200 800 20 20 0.0000000 6.2831853}%
\special{pn 8}%
\special{ia 1200 800 20 20 0.0000000 6.2831853}%
%
\special{sh 1.000}%
\special{ia 800 800 20 20 0.0000000 6.2831853}%
\special{pn 8}%
\special{ia 800 800 20 20 0.0000000 6.2831853}%
\end{picture}}%
\]
\caption{The fan with the lattice points on the left and the polytope $P_D$ on the right.}
\label{pic:fan and polytope}
\end{figure}
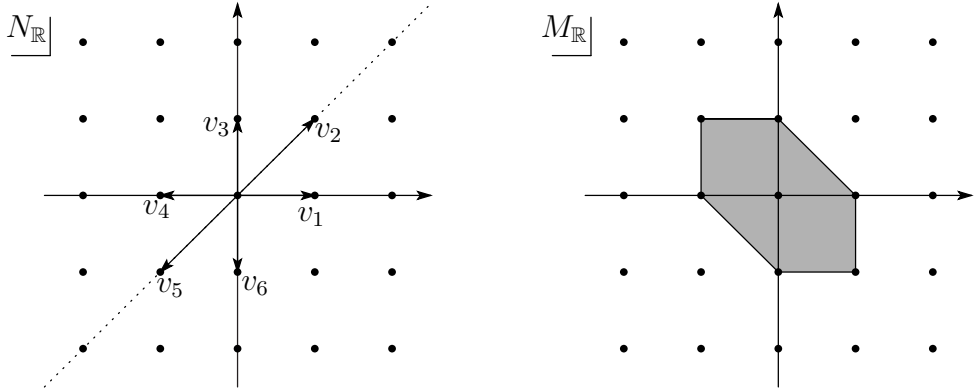 
We know that $\text{Vol}(X, L_D)=\text{Vol}(P_D),$ where $L_D(\cong K_X^{\vee})$ is the line bundle associated to the divisor $D$ and $\text{Vol}(P_D)$ is the volume of $P_D$ in $M_\R$ on which the Lebesgue measure is normalized so that each cell (i.e. each square in Figure \ref{pic:fan and polytope}) of the lattice in $M_{\R}$ has a unit volume. Hence, we deduce that $\text{Vol}(X,L_\rho)=3$. 
}
 
\emph{
On the other hand, we have 
\begin{align*}
P_{\mathfrak{sl}_3(\C)}=\frac{1}{2}\alpha_1\alpha_2(\alpha_1+\alpha_2).
\end{align*}
Hence, by Proposition \ref{prop: volume formula}, we obtain
\begin{align*}
\text{Vol}(X,L_\rho)=(\partial_{\alpha_1+\alpha_2} P_{\mathfrak{sl}_3(\C)})(\rho)=\Big(\frac{1}{2}\alpha_1^2+2\alpha_1\alpha_2+\frac{1}{2}\alpha_2^2\Big)(\varpi_1+\varpi_2)=3.
\end{align*}
Thus, the two ways of computing the volume coincide. 
}
\end{example}

\vspace{20pt}
As an application of Proposition \ref{prop: volume formula}, let us explain a relation between our volume formula and the cohomology rings of regular nilpotent Hessenberg varieties. Let $N_0$ be a regular nilpotent element of $\mathfrak{g}$ and $\Hess{N_0}{H}$ the associated regular nilpotent Hessenberg variety. In \cite{AHMMS}, it is shown that $H^*(\Hess{N_0}{H};\R)$ is a Poincar\'e duality algebra generated by degree 2 elements and that
\begin{align}\label{eq:AHMMS isom}
H^*(\Hess{N_0}{H};\R) \cong R/\text{Ann}(P_{H})
\end{align}
as graded rings over $\R$, where we have $R=\Sym(\mathfrak{t}_{\R}^*)$. The annihilator $\text{Ann}(P_{H})$ is defined as
\begin{align*}
\text{Ann}(P_{H}) = \{ F \in R \mid \partial_{F}(P_{H})=0 \},
\end{align*}
where we associate to an element $F\in R$ a differential operator $\partial_{F}$ on $R$ as a polynomial of derivations by extending  the linear association $f(\in\mathfrak{t}_{\R}^*)\mapsto \partial_{f}$ explained above \eqref{eq: def of volume formula}. Namely, the polynomial $P_{H}$ completely determines the algebra $H^*(\Hess{N_0}{H};\R)$ (see \cite[Sect.\ 1]{Kaveh} for more general treatment of Poincar\'e duality algebras). Now, Proposition \ref{prop: volume formula} shows that, for any regular dominant weight $\lambda$ of $T$, the value $P_{H}(\lambda)\in\R$ is the volume of the embedding $\Hess{N_0}{H}\hookrightarrow \P(V_{\lambda})$. In this sense, the volume formula for $\Hess{N_0}{H}$ describes the algebra structure of the cohomology ring $H^*(\Hess{N_0}{H};\R)$ via the isomorphism \eqref{eq:AHMMS isom} proven in \cite{AHMMS}.

\bigskip
\section*{Appendix}
\setcounter{equation}{0}
\setcounter{theorem}{0}
\renewcommand{\thesection}{A}

Recall that $I$ is the Dynkin diagram of $\mathfrak{g}$. For $J\subseteq I$ (regarded as a full-subgraph), let $x_J=s_J+\sum_{i\in J}E_{\alpha_i}\in\mathfrak{g}_{{\rm reg}}$ be a regular element of the form \eqref{eq:Jordan decomp} and $N_0\coloneqq \sum_{i\in I}E_{\alpha_i}$ the regular nilpotent element of the form \eqref{eq:Jordan decomp}. Namely, $s_J$ is an element of $\mathfrak{t}$ such that
\begin{align}\label{eq:app centralizer of s_J}
Z_{\mathfrak{g}}(s_J) = \mathfrak{g}(J) \oplus Z,
\end{align}
where $\mathfrak{g}(J)$ is the complex semisimple Lie algebra corresponding to the Dynkin diagram $J$, and $Z\subseteq\mathfrak{t}$ is the center of $Z_{\mathfrak{g}}(s_J)$. Let us take a line in $\mathfrak{g}$ given by
\begin{align*}
x(t)\coloneqq tx_J+(1-t)N_0 \quad (t\in\C).
\end{align*}
We prove the following claim which we used in Section \ref{subsec:degenerations}.
\begin{lemma}\label{lem: app main claim}
For any fixed $t\neq0$, the element $x(t)$ is conjugate to $tx_J$.
\end{lemma}

To prove this, we rather start with the following.
\begin{lemma}\label{lem: app subclaim}
For $x=s_J+\sum_{i\in I}c_{\alpha_i}E_{\alpha_i}$ with arbitrary coefficients $c_{\alpha_i}$, there exists an element $g\in G$ such that  
\begin{align*}
{\rm Ad}(g)x = s_J+\sum_{j\in J}c_{\alpha_j}E_{\alpha_j}.
\end{align*}
\end{lemma}

\begin{proof}
If $I\setminus J$ is empty, then there is nothing to prove. So we can assume that there is an element $\ell\in I\setminus J$. This means that $\alpha_{\ell}$ is not a root of $\mathfrak{g}(J)$ because of \eqref{eq:app centralizer of s_J}. Hence, we have $E_{\alpha_{\ell}}\notin Z_{\mathfrak{g}}(s_J)$ which implies $\alpha_{\ell}(s_J)\neq0$. By expanding the exponential, we have
\begin{align*}
\exp \left(\text{ad}\Big(\frac{c_{\alpha_{\ell}}}{\alpha_{\ell}(s_J)}E_{\alpha_{\ell}}\Big)\right) s_J
= s_J - c_{\alpha_{\ell}}E_{\alpha_{\ell}} .
\end{align*}
So it follows that
\begin{align}\label{eq:app 50}
\exp \left(\text{ad}\Big(\frac{c_{\alpha_{\ell}}}{\alpha_{\ell}(s_J)}E_{\alpha_{\ell}}\Big)\right) x
= \Big(s_J+\sum_{i\in I\setminus\{\ell\}}c_{\alpha_i}E_{\alpha_i} \Big)
+ \sum_{\alpha\succ\alpha_{\ell}} c'_{\alpha} E_{\alpha}
\end{align}
for some coefficients $c'_{\alpha}\in\C$, where $\succ$ is the partial order on $\Phi^+$ introduced in Section \ref{sec: regular Hessenberg varieties}. Here, $\alpha$ satisfying $\alpha\succ\alpha_{\ell}$ is not a root of $\mathfrak{g}(J)$ since $\alpha_{\ell}$ is not. If $(I\setminus J)\setminus\{\ell\}$ is not empty, then we take an element $\ell'\in (I\setminus J)\setminus\{\ell\}$. Applying a similar operator 
\begin{align*}
\exp \left(\text{ad}\Big(\frac{c_{\alpha_{\ell'}}}{\alpha_{\ell'}(s_J)}E_{\alpha_{\ell'}}\Big)\right)
\end{align*}
to the right-hand side of \eqref{eq:app 50}, we obtain
\begin{align*}
\Big(s_J+\sum_{i\in I\setminus\{\ell,\ell'\}}c_{\alpha_i}E_{\alpha_i} \Big)
+ \sum_{\alpha\succ\alpha_{\ell}} c'_{\alpha} E_{\alpha}
+ \sum_{\beta\succ\alpha_{\ell'}} c''_{\beta} E_{\beta}
\end{align*}
for some coefficients $c''_{\beta}\in\C$. By the same reason as above, $\beta$ satisfying $\beta\succ\alpha_{\ell'}$ is not a root of $\mathfrak{g}(J)$. By repeating this procedure, we see that there exists an element $g_1\in G$ such that 
\begin{align*}
\text{Ad}(g_1) x
= \Big(s_J+\sum_{i\in J}c_{\alpha_i}E_{\alpha_i} \Big)
+ \sum_{\text{ht}(\alpha)>1, \ \alpha\notin\Phi(J)} d_{\alpha} E_{\alpha}
\end{align*}
for some coefficients $d_{\alpha}\in\C$, where $\Phi(J)$ is the set of roots of $\mathfrak{g}(J)$. Since $\alpha$ appearing in the summation in the right-hand side is not a root of $\mathfrak{g}(J)$, we can repeat the whole procedure again, and we see that there exists an element $g_2\in G$ such that 
\begin{align*}
\text{Ad}(g_2) x
= \Big(s_J+\sum_{i\in J}c_{\alpha_i}E_{\alpha_i} \Big)
+ \sum_{\text{ht}(\alpha)>2, \ \alpha\notin\Phi(J)} d'_{\alpha} E_{\alpha}
\end{align*}
for some coefficients $d'_{\alpha}\in\C$. Continuing this, one can prove by induction that there exists an element $g\in G$ such that 
\begin{align*}
\text{Ad}(g) x
= s_J+\sum_{i\in J}c_{\alpha_i}E_{\alpha_i},
\end{align*}
as desired.
\end{proof}

\begin{proof}[Proof of Lemma {\rm \ref{lem: app main claim}}]
By the definition of $x(t)$, we have  
\begin{align*}
x(t) = ts_J + \sum_{j\in J}E_{\alpha_j} + \sum_{\ell\in I\setminus J}(1-t)E_{\alpha_\ell}.
\end{align*}
Suppose that $t\neq0$. Then by Lemma \ref{lem: app subclaim}, there exists an element $g'\in G$ such that 
\begin{align*}
\text{Ad}(g')x(t) = ts_J + \sum_{j\in J}E_{\alpha_j}.
\end{align*}
Recalling that the set of simple roots $\Delta=\{\alpha_1,\ldots,\alpha_n\}$ is a $\C$-basis of $\mathfrak{t}^*$, let $\{\epsilon_1,\ldots,\epsilon_n\}$ $\subset$ $\mathfrak{t}$ be its dual basis. Since $t\neq0$, we can take $c\in\C$ such that $e^c=t$. Now let 
\begin{align*}
s = \sum_{j\in J} c\epsilon_j \in\mathfrak{t}.
\end{align*}
Then 
\begin{align*}
\exp(\text{ad}(s))\Big(ts_J + \sum_{j\in J}E_{\alpha_j} \Big)=ts_J + \sum_{j\in J}tE_{\alpha_j}
\end{align*}
since $s_J\in\mathfrak{t}$. Thus, we conclude that there exists an element $g\in G$ such that 
\begin{align*}
\text{Ad}(g)x(t) = 
ts_J + \sum_{j\in J}tE_{\alpha_j}
\end{align*}
which is exactly $tx_J$. 
\end{proof}

\end{document}